\documentclass[sort&compress,preprint,11pt]{elsarticle}

\usepackage{amssymb}
\usepackage{amsthm}

\usepackage[normalem]{ulem}
\usepackage{lineno}
\usepackage{amsmath}
\usepackage{url}
\usepackage{lipsum}
\usepackage{amsfonts}
\usepackage{graphicx}
\usepackage{epstopdf}
\usepackage{algorithmic}
\usepackage{colortbl}
\usepackage{comment}
\usepackage[table]{xcolor}
\usepackage[skip=1pt]{caption} 
\usepackage{hyperref}
\newcommand{\stkout}[1]{\ifmmode\text{\sout{\ensuremath{#1}}}\else\sout{#1}\fi}
\allowdisplaybreaks

\usepackage{subcaption}
\ifpdf
  \DeclareGraphicsExtensions{.eps,.pdf,.png,.jpg}
\else
  \DeclareGraphicsExtensions{.eps}
\fi

\newcommand{\R}{\mathbb R}

\newcommand{\bfe}{\ensuremath{\mathbf{e}}}
\newcommand{\bff}{\ensuremath{\mathbf{f}}}
\newcommand{\bfw}{\ensuremath{\mathbf{w}}}
\newcommand{\bfx}{\ensuremath{\mathbf{x}}}

\newcommand{\bfu}{\ensuremath{\mathbf{u}}}
\newcommand{\bfv}{\ensuremath{\mathbf{v}}}

\newcommand{\bfX}{\ensuremath{\mathbf{X}}}
\newcommand{\bfH}{\ensuremath{\mathbf{H}}}
\newcommand{\bfL}{\ensuremath{\mathbf{L}}}
\newcommand{\bfV}{\ensuremath{\mathbf{V}}}

\newcommand{\bfphi}{\ensuremath{\boldsymbol{\phi}}}
\newcommand{\bvphi}{\ensuremath{\boldsymbol{\varphi}}}
\newcommand{\bfeta}{\ensuremath{\boldsymbol{\eta}}}

\newcommand{\baru}{\ensuremath{\Bar{\bfu}}}

\newcommand{\ur}{\ensuremath{\mathbf{u}_r}}
\newcommand{\vr}{\ensuremath{\mathbf{v}_r}}
\renewcommand{\wr}{\ensuremath{\mathbf{w}_r}}
\newcommand{\phir}{\ensuremath{\bfphi_r}}

\newcommand{\Rey}{\ensuremath{\mathrm{Re}}}
\newcommand{\ohm}{\Omega}

\newcommand{\inv}{^{-1}}
\newcommand{\delt}{\ensuremath{\Delta t}}

\newcommand{\divergence}{\nabla \cdot}

\newcommand{\Grad}{\ensuremath{\nabla}}

\newcommand{\omint}{\int_{\Omega}}

\newcommand{\pare}[1]{\left({}#1\right)}
\newcommand{\curly}[1]{\left\{{}#1\right\}}

\newcommand{\abs}[1]{\ensuremath{\left\lvert{#1}\right\rvert}}
\newcommand{\norm}[1]{\ensuremath{\left\|{#1}\right\|}}
\newcommand{\trinorm}[1]{{\left\vert\kern-0.15ex\left\vert\kern-0.15ex\left\vert #1 
    \right\vert\kern-0.15ex\right\vert\kern-0.15ex\right\vert}}

\newcommand{\be}{\begin{equation}}
\newcommand{\ee}{\end{equation}}
\newcommand{\bes}{\begin{equation*}}
\newcommand{\ees}{\end{equation*}}

\def\al#1\eal{\begin{align}#1\end{align}}
\def\als#1\eals{\begin{align*}#1\end{align*}}

\newcommand{\ti}[1]{{ \color{blue}{[TI - #1]}}}
\newcommand{\jr}[1]{{ \color{teal}{[JR - #1]}}}
\newcommand{\pt}[1]{{ \color{violet}{[PHT - #1]}}}

\usepackage[normalem]{ulem}

\def\bu{{\bf u}}
 
\def\uur{{\underline {u}}_r} 
\def\buur{{\underline {\overline{u}}}_r} 
 
\def\uf{{\underline {f}}_r} 
\def\cN{{\cal N}} 
\def\bX{{\bf X}}
\def\bx{{\bf x}}

\def\bv{{\bf v}}
\newcommand{\bphi}{\boldsymbol{\varphi}}
 
\def\mO{\mathcal{O}}  
\def\Ll2{\Lambda^r_{L^2}}  
\def\Lh10{\Lambda^r_{H^1_0}}  
\def\sqrtLl2h10{\sqrt{\Ll2\Lh10}}  
\def\epl2{\varepsilon_{L^2}}
\def\eph10{\varepsilon_{H^1_0}}
\def\epavgh10{\varepsilon^{avg}_{H^1_0}}
\def\efchi{\chi_{\text{effective}}}
\def\thechi{\chi_{\text{theory}}}
\def\enedelta{\delta_{\text{energy}}}

\newtheorem{theorem}{Theorem}[section]

\newtheorem{corollary}{Corollary}[theorem]
\newtheorem{lem}{Lemma}[section]
\newtheorem{lemma}{Lemma}[section]

\newtheorem{remark}{Remark}[section]
\newtheorem{definition}{Definition}[section]

\newtheorem{assumption}[lem]{Assumption}

\journal{ ArXiv }

\begin{document}

\begin{frontmatter}

\title{
{\it A Priori} Error Bounds and Parameter Scalings for
the \\ Time Relaxation 
Reduced Order Model}


\author[VT]{Jorge Reyes\corref{cor1}}
\author[VT]{Ping-Hsuan Tsai}
\author[UAM]{Julia Novo}
\author[VT]{Traian Iliescu}

\cortext[cor1]{Corresponding author}

\affiliation[VT]{organization={Department of Mathematics},
            addressline ={ Virginia Tech}, 
            city={ Blacksburg},
            postcode={24060}, 
            state={VA},
            country={USA}}

\affiliation[UAM]{organization={Department of Mathematics},
            addressline ={Universidad Autonoma de Madrid}, 
            city={ Madrid},
            postcode={24060}, 
            state={VA},
            country={Spain}}

\begin{abstract}
The {\it a priori} error analysis of reduced order models (ROMs) for fluids is relatively scarce. In this paper, we take a step in this direction and conduct numerical analysis of the recently introduced time relaxation ROM (TR-ROM), which uses spatial filtering to stabilize ROMs for convection-dominated flows. Specifically, we prove stability, an {\it a priori} error bound, and parameter scalings for the TR-ROM. Our numerical investigation shows that the theoretical convergence rate and the parameter scalings with respect to ROM dimension and filter radius are recovered numerically. In addition, the parameter scaling can be used to extrapolate the time relaxation parameter to other ROM dimensions and filter radii. Moreover, the parameter scaling with respect to filter radius is also observed in the predictive regime.
\end{abstract}


\end{frontmatter}

\section{Introduction}


The incompressible Navier-Stokes equations (NSE) are 
\begin{align}
\bfu_t+ (\bfu \cdot \nabla) \bfu -\nu \Delta \bfu+\nabla p &=\bff,\label{eq:Strong_NSE1}\\
\divergence \bfu&=0,\label{eq:Strong_NSE2}
\end{align}
where $\bfu$ 
and $p$ 
are the velocity and pressure fields, respectively, defined on the spatial domain, $\ohm$, and the time interval, $[0,T]$. $\bff$ is an external force, and $\nu$ 
is the inverse of the Reynolds number. 
Appropriate boundary and initial conditions are needed to close the system. 

Fluid flows 
at high Reynolds numbers exhibit 
a wide range of spatial and temporal scales 
that make their direct numerical simulation (DNS) 
often impractical \cite{BIL06,layton2012approximate}. 
This leads to 
the need of alternative computational approaches, 
such as large eddy simulations (LES), Reynolds-averaged Navier–Stokes equations (RANS), and numerical regularizations. 
One type of regularization is the 
time relaxation model (TRM) 
\cite{stolz2001approximate,stolz2001shocks}, which leverages spatial filtering to increase the numerical stability. The TRM 
for a domain $ \ohm \subset \R^d$, $ d= 2 \text{ or } 3$, and for $ t >0 $ 
is given as 
\al
\bfu_t+ (\bfu \cdot \nabla) \bfu-\nu \Delta \bfu + \chi\bfu^* +\nabla p &=\bff,\label{eq:Strong_TRM1}\\
\divergence \bfu&=0,\label{eq:Strong_TRM2}
\eal
where the dimensionless parameter $ \chi $ is called the \emph{time relaxation parameter}, which is often manually tuned to 
adjust the numerical stabilization, and $ \bfu^* $ is a regularization term defined 
in Section \ref{sec: Notation}.
The goal of $ \bfu^* $ is to drive the unresolved fluctuations of $ \bfu$ down to $ 0$.
TRM has been investigated in 
~\cite{belding2022efficient,belding2020computational,dunca2014numerical,ervin2007numerical,layton2007truncation,neda2012increasing,takhirov2018modular}
and has been used in various applications~\cite{neda2016finite,takhirov2016time}.
A TRM 
review can be found in \cite{breckling2017review}. 

Although the DNS computational cost is significantly reduced by LES, RANS, and numerical stabilization, these  approaches remain computationally prohibitive in decision-making applications where multiple forward simulations are needed.
In those cases, reduced order models (ROMs) represent efficient alternatives.  
ROMs are computational models 
whose dimension is orders of magnitude lower
than the dimension of 
full order models (FOMs), i.e.,  models obtained from
classical numerical discretizations.  
In the numerical simulation of fluid flows, Galerkin ROMs
(G-ROMs), which use data-driven basis functions in a Galerkin framework, have provided
efficient and accurate approximations of laminar flows, such as the
two-dimensional flow past a circular cylinder at low Reynolds
numbers~\cite{
hesthaven2015certified,
quarteroni2015reduced}.
However, turbulent flows are notoriously hard for the
standard G-ROM. 
Indeed, to capture the complex dynamics, 
a large number 
\cite{tsai2023accelerating} of ROM basis functions is required, 
which yields high-dimensional ROMs that cannot be used in realistic applications.
Thus, computationally efficient, low-dimensional ROMs are used instead.
Unfortunately, these ROMs are inaccurate since 
the ROM basis functions that were not used 
to build the G-ROM 
have an important role 
in dissipating the energy from the system~\cite{ahmed2021closures}.
Indeed, without 
enough dissipation, 
the low-dimensional G-ROM generally yields spurious numerical oscillations.  
Thus, closures and stabilization strategies are required for the low-dimensional G-ROMs to be stable and accurate~\cite{ahmed2021closures,
fick2018stabilized,kaneko2022augmented,kaneko2020towards,mou2023energy,parish2024residual}. 

FOM 
stabilizations and closures 
are supported by thorough 
numerical analysis, particularly when applied alongside traditional methods like the finite element method (FEM) or SEM~\cite{BIL06,john2004large,rebollo2014mathematical,roos2008robust}. 
These references address both fundamental numerical analysis issues, such as stability and convergence, and practical challenges, like determining appropriate parameter scalings for stabilization coefficients. These two aspects are closely linked, as insights from numerical analysis guide the selection of parameter scalings, 
which inform practical decisions. 
We emphasize, however, that despite growing interest in ROM closures 
and stabilizations, 
their comprehensive mathematical and numerical analysis remains an open challenge. 
Indeed, while some strides have been made in analyzing 
ROM closures and stabilizations~\cite{giere2015supg,novo2021error,gunzburger2020leray,iliescu2013variational,iliescu2014variational,john2022error,xie2018numerical}, much work is needed to reach the rigor of FOM analysis.

In this paper, we take a step 
in this direction by establishing the first rigorous numerical analysis results, including stability and \textit{a priori} bounds, for the time relaxation reduced order model (TR-ROM), 
which was successfully used in \cite{tsai2023time} in numerical simulations of turbulent channel flow. 
Crucially, we also derive parameter scalings that ensure ROM parameters automatically adjust with changes in the corresponding FOM and ROM parameters, eliminating the need for manual tuning often required in existing data-driven ROMs.

This article is organized as follows: 
In Section \ref{sec: Notation}, we give preliminaries about the SEM, G-ROM, and ROM filtering.
In Section \ref{sec: main}, we present the TR-ROM, prove its unconditional stability and an {\it a priori} error bound, and derive novel scalings for the time relaxation parameter.
In Section \ref{sec: Num_Sims}, we show that the theoretical convergence rates and parameter scalings with respect to ROM dimension and filter radius are numerically recovered for the 2D flow past a cylinder and 2D lid-driven cavity.
In Section \ref{sec: con}, we present the
conclusions of our theoretical and numerical investigations.

\section{Notations and Preliminaries}
\label{sec: Notation}
\subsection{Spectral Element Method}
\label{ssec: Sem Intro} 


This paper will use the following spaces: $L^p(\Omega) $, $W^{k,p}(\Omega) $, and $ H^k(\Omega) = W^{k,2}(\Omega)$, where $ k \in \mathbb{N}, 1 \leq p \leq \infty $ for domain $\Omega$.
 The $L^2(\Omega)$ norm is denoted as $\Vert \cdot \Vert$, with the corresponding inner product $(\cdot,\cdot)$. 
 Vector-valued functions are indicated in boldface having $d$ components ($d=2$ or $3$).
 The $H^k(\ohm)$ norm will be denoted by $ \norm{\cdot}_k$, with all other norms clearly denoted. For the continuous vector function $ \bfu(\bfx,t)$ defined on the entire time interval $ (0,T)$, we have
\begin{equation*}
 \norm{\bfu}_{\infty,k} := \sup_{0<t<T} \norm{\bfu(\cdot,t)}_k, \qquad \text{ and } \qquad \norm{\bfu}_{m,k} := \pare{\int_0^T \norm{\bfu(\cdot,t)}_k^m \, dt}^{\frac{1}{m}}.   
\end{equation*}
 The solutions are sought in the following functional spaces:
\begin{eqnarray*}
    & &\mathrm{Velocity\;space} -\bX:= \bfH_0^1(\Omega) = \curly{ \bfu \in \bfH_0^1(\ohm): \bfu\mid_{\partial\ohm} =0 }\, , \\
    & &\mathrm{Pressure\;space} - Q:=L^{2}_{0}(\Omega)=\left\{q \in L^2(\Omega) :
     \int_\Omega \,  q \, d\Omega = 0 
     \right\} \, ,\\
    & &\mathrm{Divergence-free\;space} - \bfV:=\left\{\bfv \in \bfX: \int_\Omega q \, \Grad\cdot
    \bfv  \, d\Omega \, = \, 0, \; \forall  q \in Q \right\}.
\end{eqnarray*}
Boldface indicates that the space is spanned by vector-valued functions. The dual space of $\bX$ is denoted as $\bfX'$, and the norm of the space is $\| \cdot \|_{-1} $.
Moreover, we define $ \bfV $ to be the weakly divergence-free subspace of $\bX$.

The FOM is based on the SEM in the open-source code Nek5000 \cite{fischer2008nek5000}, and uses the
$P_{N}$--$P_{N-2}$ 
velocity-pressure coupling \cite{maday1989spectral}. 
To this end, let $\Omega $ be a polygonal domain and $\left\{\Omega_i\right\}_{i=1}^K$ be a conforming partition
of $\Omega$ into rectangles or rectangular parallelepipeds. If $\bfu$ is a function defined in
$\overline{\Omega}= \cup_{i=1}^K \overline{\Omega}_i$, the
restriction of $\bfu$ to $\overline{\Omega}_i$ will be denoted by
$\bu_i$. We set
$$
\bfX_N=\left\{\bfv_N\in (C^0(\overline{\Omega}))^d\mid (\bfv_N)_i\in
({\mathbb P}_N(\overline{\Omega}_i))^d,\ (\bfv_N)_i=0 \ {\rm in}\
\partial\Omega\cap\partial\Omega_i,\ i=1,\ldots,K\right\},
$$
where ${\mathbb P}_N$ denotes the space of polynomials of degree less or equal to $N$ with
respect to each variable. 
 The discrete space of pressures $Q_N \subset Q$ is defined as
$$
Q_N=\left\{q_N\in L_0^2(\Omega)\mid (q_N)_i \in {\mathbb
P}_{N-2}(\overline{\Omega}_i), \ i=1,\ldots,K\right\}.
$$
The discrete velocity belongs to the space
$$
\bfV_N =\left\{\bfv_N\in \bfX_N \mid  (\nabla\cdot
\bfv_N,q_N)=0 \quad \forall q_N\in Q_N\right\}.
$$
With the above choice for the space of pressures, the following
$\inf$-$\sup$ condition is satisfied  \cite{be_et_al,ber_ma},
\begin{equation}
\inf_{q_N\in Q_N} \sup_{\bfv_N\in \bfX_N}{{\left(q_N,\nabla\cdot
\bv_N\right)}\over{\| q_N\| \| \nabla \bfv_N\|}}\ge {\beta}N^{\frac{1-d}{2}},
\label{inf-sup}
\end{equation}
where $\beta$ is a constant that does not depend on $N$.

Let $\delt$ denote the time step, 
and $t^{n} = n \delt$, $n = 0, 1, \dots, M$, the time instances. 
We also use the notation $ \bfu(t^n) = \bfu^n$
and the following discrete norms:  
\begin{eqnarray*}
 \trinorm{ \bfu }_{\infty, k} := \max_{0 \le n \le M} \|\bfu^{n}\|_{k} \, ,
   \hspace{1 cm}
    \trinorm{ \bfu }_{m,k} := \left(\delt\sum_{n=0}^{M} \| \bfu^{n} \|^{m}_{k}\right)^{1/m} .
\end{eqnarray*}
For $ \bfu,\bfv,\bfw \in \bX $, we define the trilinear forms $ b, b^*: \bX \times \bX \times \bX \mapsto \R  $ 
as follows: 
\begin{eqnarray}
 b(\bfu,\bfv,\bfw) &=& (\bfu\cdot\nabla \bfv,\bfw), \label{eq: B_term} \\
 b^{*}(\bfu,\bfv,\bfw)&=& (\bfu \cdot \Grad\bfv ,\bfw) + \frac{1}{2}((\Grad\cdot \bfu)\bfv,\bfw) = \frac{1}{2}(b(\bfu,\bfv,\bfw)-b(\bfu,\bfw,\bfv)). \label{eq: Bstar_term}
 \end{eqnarray}
The following approximation properties hold \cite{be_et_al,ber_ma}:
\al
    \inf_{\bfv_N\in \bX_{N}}\| \bfu - \bfv_N \| &\le C N^{-k-1} \| \bfu \|_{{k+1}},\;\; \bfu \in
    \bfH^{k+1}(\Omega)^{d}, \label{prop1} \\
    \inf_{\bfv_N\in \bX_{N}}\|\Grad( \bfu - \bfv_N) \| &\le C N^{-k}\| \bfu \|_{{k+1}},\;\; \bfu \in
    \bfH^{k+1}(\Omega)^{d}, \label{prop2} \\
    \inf_{q_N \in Q_{N}} \| p - q_N \| &\le CN^{-s-1} \| p \|_{s+1},\;\;  p \in
    H^{s+1}(\Omega). \label{prop3}
\eal
We also use the following lemmas:\begin{lemma}\label{TRIL} \cite{ Layton08, temam2001navier},
For $\bfu, \bfv, \bfw \in \bX$, 
$b^{*}(\bfu, \bfv, \bfw)$ can be bounded 
as follows: 
\al
 b^{*}(\bfu, \bfv, \bfw) &\leq   C \norm{\bfu}^{\frac{1}{2}} \norm{\Grad \bfu}^{\frac{1}{2}} \left\|\Grad \bfv\right\|  \left\|\Grad \bfw\right\|  , \\
 b^{*}(\bfu, \bfv, \bfw) &\leq  C \left\|\Grad \bfu\right\| \left\|\Grad \bfv\right\|  \left\|\Grad \bfw\right\| ,\\
 b^*(\bfu,\bfv,\bfw) &\leq C \norm{\Grad \bfu} \norm{\Grad \bfv} \norm{\bfw}^{1/2} \norm{\Grad \bfw}^{1/2}.\label{eq: bbound}
\eal

\end{lemma}

\begin{lemma}[Discrete Gronwall Lemma \cite{heywood1990finite}]
\label{discreteGronwall} Let $\Delta t$, H, and $a_{n},b_{n},c_{n},d_{n}$
(for integers $n \ge 0$) be finite nonnegative numbers such that
\begin{equation*}
a_{l}+\Delta t \sum_{n=0}^{l} b_{n} \le \Delta t \sum_{n=0}^{l} d_{n}a_{n} +
\Delta t\sum_{n=0}^{l}c_{n} + H \ \ for \ \ l\ge 0. \label{gronwall1}
\end{equation*}
Suppose that $\Delta t d_n < 1 \; \forall n$. Then,
\begin{equation*}
a_{l}+ \Delta t\sum_{n=0}^{l}b_{n} \le \exp\left( \Delta t\sum_{n=0}^{l} \frac{d_{n}}{1 - \Delta t d_n } \right) \left( \Delta t\sum_{n=0}^{l}c_{n} + H
\right)\ \ for \ \ l \ge 0.
\end{equation*}
\end{lemma}

Let $\bfu$ be the velocity of the Navier-Stokes equations for a given initial condition, and let $\bfu_N$ be its continuous
in time spectral element approximation. Then, the following bound for the error is proved in \cite{frutos_novo}  (see Remarks 4.2 and 4.3): 
$$
\|\bfu(t)-\bfu_N(t)\|+(N)^{-1}\|\nabla(\bfu(t)-\bfu_N(t))\|\le C N^{-(k+1)},\ t\in(0,T],
$$
assuming $\bfu(t)\in \bfH^{k+1}(\Omega)$ 
($d=2$).

Using standard techniques, one can also prove error bounds for the fully discrete method. 
In particular, we will make the following assumption:  
\begin{assumption}[{Spectral Element Error}]\label{as: FEMassumption}
    Under sufficient regularity of the true solution, we assume that the fully discrete approximation of \eqref{eq:Strong_NSE1}-\eqref{eq:Strong_NSE2} using SEM in space and [BDF3/EXT3] in time, $ \bfu_N^n \in \bX_N$ for $ 1\le n\le M$, satisfies the following asymptotic error estimate:
    \begin{equation*}
        \norm{\bfu^{n} - \bfu_N^n}^2 \leq C(N^{-{2k-2}}+ \delt^6),\qquad    
        \norm{\Grad\pare{\bfu^n - \bfu_N^n}}^2 \leq C(N^{-{2k}} + \delt^6).
    \end{equation*}
\end{assumption}
\begin{remark}
    As discussed in \cite{fischer2017recent}, $k=3$ is used to
    ensure that the imaginary eigenvalues associated with skew-symmetric advection
    operator are  within the stability region of the BDF$3$/EXT$3$ time-stepper. 
\end{remark}

\subsection{Galerkin 
Reduced Order Model (G-ROM)}

In this section, we introduce the G-ROM
. We follow the standard proper
orthogonal decomposition (POD) procedure
\cite{berkooz1993proper,volkwein2013proper} to construct the reduced basis function.  To
this end, we collect a set of spectral element (FOM) solutions lifted by the zeroth mode $\bphi_0$. The POD method seeks a low-dimensional basis $\{\bphi_1,\ldots,\bphi_r\}$ in $\bfL^2$ that optimally approximates the snapshots, that is, solves the minimization problem: 
\begin{equation*}
    \min \frac{1}{M+1}\sum^M_{l=0}\norm{\bu{_N}(\cdot,t_l)- \sum^r_{j=1}\left(\bu{_N}(\cdot,t_l),\bphi_j(\cdot)\right)\bphi_j(\cdot)}^2
\end{equation*}
subject to the conditions $(\bphi_i,\bphi_j)= \delta_{ij}$, for $1\le i,j \le r$, where $\delta_{ij}$ is the Kronecker delta. The minimization problem can be solved 
by considering the eigenvalue
problem $\mathcal{K}\underline{z}_j = \lambda_j \underline{z}_j$, for $j=1,\ldots,r$,
where $\mathcal{K} \in \mathbb{R}^{(M+1)\times(M+1)}$ 
is the snapshot Gramian matrix using
the $L^2$ inner product (see, e.g.,~\cite{kaneko2020towards,tsai2022parametric}
for alternative strategies). 

The first $r$ POD basis functions $\{\bphi_i\}^r_{i=1}$ are constructed from the first $r$ eigenmodes of the Gramian matrix. The G-ROM is then constructed by inserting the POD approximated solution 
$
   \bu_r(\bx) = \sum_{j=1}^r u_{r,j} \bphi_j(\bx)
$
into the weak form of the NSE:
{\em Find $\bu_r$ 
such that, for all $\bv_r \in \bX_r$,}
\begin{eqnarray}
    && 
    \left(
        \frac{\partial \bu_r}{\partial t} , \bv_r 
    \right)
    + Re^{-1} \, 
    \left( 
        \nabla \bu_r , 
        \nabla \bv_r 
    \right)
    + \biggl( 
        (\bu_r \cdot \nabla) \bu_r ,
        \bv_r 
    \biggr)
    = 0,  
    \label{eq:gromu}
\end{eqnarray}
where 
$\bX_r := \text{span} \{\bphi_i\}^r_{i=1}$ is the ROM space. 

It can also be shown that the following error formula holds for the $L^2$-POD basis functions \cite{HLB96}:
\al\label{eq: POD_Proj_err}
\frac{1}{M+1} \sum_{\ell =0}^{M} \norm{ \bu_N(\cdot,t_\ell) - \sum_{j=1}^{r} \pare{\bfu_N(\cdot,t_\ell),\bphi_j(\cdot)} \bphi_j(\cdot) }^{2}
= \Ll2 := \sum_{j=r+1}^{R} \lambda_j,
\eal
where $R$ is the rank of the Gramian matrix, $ \mathcal{K}$.
\begin{remark}
    Because the POD basis functions are a linear combination of the snapshots generated from the FOM, the POD basis functions satisfy the boundary conditions of the original PDE and 
    inherit the FOM's divergence-free properties. 
    In this paper, the FOM is based on a SEM discretization, which yields only a weakly divergence-free velocity.
    More precisely, 
    the POD basis functions belong to $ \bfV_N$, giving $ \bX_r \subset \bfV_N$. 
    Thus, to ensure the ROM stability in Lemma~\ref{lemma: Stability}, we equip the ROM with the skew-symmetric trilinear form $b^*$ in \eqref{eq: Bstar_term}.
\end{remark}
Additionally, we make use of the following definitions and lemmas: 
\begin{definition}[ROM $L^2$ Projection]\label{def: Rom_Proj}
    Let $ P_r: \bfL^2 \to \bX_r$ such that, $ \forall~\bfu \in \bfL^2, P_r(\bfu)$ is the unique element of $ \bX_r$ satisfying 
    \begin{equation}\label{eq: ROM_Proj}
        \pare{P_r(\bfu), \vr} = \pare{\bfu,\vr}, ~ \forall~ \vr \in \bX_r.
    \end{equation}
\end{definition}
\begin{lemma}[$H_0^1$ POD Projection error]
\label{Lemma: H1POD_Proj_err}
The POD projection error in the $ H_0^1$ norm satisfies
\al
\label{eq: H10_POD_Proj_err}
\frac{1}{M+1} \sum_{\ell =0}^{M} \norm{ \bfu_N(\cdot,t_\ell) - \sum_{j=1}^{r} \pare{\bfu_N(\cdot,t_\ell),\bvphi_j(\cdot)} \bvphi_j(\cdot) }_{\bfH_0^1}^{2} = \Lh10 := \sum_{j=r+1}^{R} \norm{\bvphi_j}_{\bfH_0^1}^{2} \lambda_j. 
\eal
\end{lemma}

\begin{proof}
     These sharper bounds can be obtained by using the $H_0^1$ inner product and norm instead of the $H^1$ inner product and norm in \cite[Lemma 3.2]{iliescu2014variational}.
    \label{remark:h01-approximation-error}
\end{proof}

We list a POD inverse estimate, which will be used in what follows. Let $ S_r \in \R^{r\times r}$ with $ (S_r)_{ij} = \pare{\Grad \bvphi_{j}, \Grad \bvphi_{i}} $ be the POD stiffness matrix. Let $ \norm{\cdot}_2$ denote the matrix 2-norm. 
Since this is traditional notation, in what follows we will use the notation $\| \cdot \|_2$ both for the $H^2$ norm and for the matrix 2-norm.
It will be clear from the context which norm is used. 
\begin{lemma}[POD Inverse Estimates \cite{xie2018numerical}]
\label{lemma: ROM_inv_estimate}
    For all $ \ur \in \bX_r$, the following POD inverse estimate holds:
    \begin{equation}\label{eq: ROM_inv_ineq}
        \norm{\Grad \ur} \leq \sqrt{\norm{S_r}_2} \norm{\ur}.
    \end{equation}
\end{lemma}
The inverse estimate \eqref{eq: ROM_inv_ineq} was proved in Lemma 2 and Remark {2} in \cite{KV01}. The scaling of $ \sqrt{\norm{S_r}_2}$ with respect to $r$ was numerically investigated in Remark 3.3 in \cite{iliescu2014are} and in Remark 3.2 in \cite{giere2015supg}. 

\begin{lemma}[$L^2$ Stability of of ROM $L^2$ Projection \cite{xie2018numerical}]
    For all $\bfu \in \bfL^2 $, the ROM projection $ P_r$ satisfies
    \begin{equation}
        \norm{P_r(\bfu)} \leq \norm{\bfu }.
    \end{equation}
\end{lemma}

The following error 
bound is a slightly modified variation of 
\cite[Lemma 3.3]{iliescu2014variational}. This is due to our different 
spectral element error Assumption \ref{as: FEMassumption} and different Lemma \ref{Lemma: H1POD_Proj_err}. 
Furthermore, we assume a third-order in time discretization 
(i.e., BDF3/EXT3) as opposed to the first-order backward Euler method.

\begin{lemma}[Modified Lemma 3.3 in \cite{iliescu2014variational}]\label{lemma: estimationprop}
    For any $ \bfu^n \in \bX$, $ n = 0, 1, \dots , M, $
    its $ L^2 $ projection, $ P_r(\bfu^n)$, satisfies the following error bounds: 
    \begin{gather}
        \frac{1}{M+1}\sum_{n=0}^{M}\norm{\bfu^n - P_r(\bfu^n)}^2 \leq C \pare{ N^{-2k-2} + \delt^6 + \Ll2 },\label{eq: L2_ROM_estimate}\\
        \hspace{-0.41 cm}\frac{1}{M+1}\sum_{n=0}^{M}\norm{\Grad\pare{\bfu^n - P_r(\bfu^n)}}^2 \leq C \pare{ N^{-2k} + \norm{S_r}_2 N^{-2k-2} +  (1 + \norm{S_r}_2 )\delt^6 + \Lh10 }\label{eq: H1_ROM_estimate}.
    \end{gather}
\end{lemma}
\begin{proof}
The proof of Lemma \ref{lemma: estimationprop} follows a similar approach to that in \cite[Lemma 3.3]{iliescu2014variational}, but here we use the spectral element Assumption \ref{as: FEMassumption} in place of the finite element error assumption.
\end{proof}

A generalization of Lemma \ref{lemma: estimationprop} is given by Corollary \ref{corr: gen_rom_error}. This allows 
a modularity of the projection error to accommodate different discretizations of the FOM. 
\begin{corollary}\label{corr: gen_rom_error}
    For any $ \bfu^n \in \bfX$, $ n = 0, 1, \dots , M, $
    its $ L^2 $ projection, $ P_r(\bfu^n)\in X_r $,
    satisfies the following error bounds: 
    \al
        \frac{1}{M+1}\sum_{n=0}^{M}&\norm{\bfu^n - P_r(\bfu^n)}^2 \leq C \pare{ \norm{\bfu^n - \bfu^n_{FOM}}^2 + \Ll2 },\label{eq: Gen_L2_ROM_estimate}\\
        \frac{1}{M+1}\sum_{n=0}^{M}&\norm{\Grad\pare{\bfu^n - P_r(\bfu^n)}}^2 \leq C \pare{ \norm{\Grad(\bfu^n - \bfu^n_{FOM}) }^2 + \norm{S_r}_2 \norm{\bfu^n - \bfu^n_{FOM}}^2  + \Lh10 }\label{eq: Gen_H1_ROM_estimate}.
    \eal
    where $ \bfu^n_{FOM}$ is the 
    solution given by the full order model (e.g. FEM, SEM, etc).
\end{corollary}

We also assume the following bounds, analogous to those in \cite{iliescu2014are}:
\begin{assumption}\label{as: Rom Error Assumption}
    For any $ \bfu^n \in \bX$, where $ n = 0, 1, \dots , M $,
    its $ L^2$ projection, $ P_r(\bfu^n) \in \bfX_r$, satisfies the following error estimates:
    \al
       \norm{\bfu^n - P_r(\bfu^n)}^2 &\leq C \pare{ N^{-2k-2} + \delt^6 + \Ll2  },\label{eq: L2_assumption} \\
       \norm{\Grad\pare{\bfu^n - P_r(\bfu^n)}}^2 &\leq C \pare{ N^{-2k} + \norm{S_r}_2 N^{-2k-2} +  (1 + \norm{S_r}_2 )\delt^6 + \Lh10 }.\label{eq: H1_assumption}
    \eal
\end{assumption}

\begin{remark}[See also Remark 3.1 in~\cite{moore2024numerical}.]

    The pointwise in time error bounds in Assumption~\ref{as: Rom Error Assumption} are needed in the proof of Theorem~\ref{thm:ErrorEstimate}.
    Specifically, we use those bounds 
    to prove inequalities~\eqref{eq: sharper} and 
    \eqref{eq: 2stabeta}.
    We emphasize that using instead the average error bounds in Lemma~\ref{lemma: estimationprop} in the proof of Theorem~\ref{thm:ErrorEstimate} would yield suboptimal error bounds (see also~\cite[Remark 3.1 and Lemma 4.2]{moore2024numerical}).

    Assumption~\ref{as: Rom Error Assumption} and its important effect on the optimality of {\it a priori} error bounds was carefully discussed in~\cite{koc2021optimal} (see also~\cite[Remark 3.2]{iliescu2014variational}).
    In particular, it was shown in~\cite{koc2021optimal} that using both the snapshots and the snapshot difference quotients to construct the ROM basis yields optimal error bounds without making Assumption~\ref{as: Rom Error Assumption}.
    This result was further improved in~\cite{locke_singler,nos_letters}, where optimal error bounds were proven using only the snapshot difference quotients and the snapshot at the initial time or the mean value of the snapshots.
    Further improvements were recently presented in~\cite{garcia2024pointwise}.
    
    For simplicity, in this paper we do not include the snapshot difference quotients, and instead assume the pointwise in time error bounds in Assumption~\ref{as: Rom Error Assumption}.
    \label{remark:pointwise-error-assumption}
\end{remark}

\subsection{ROM filtering}
We 
formally introduce the time relaxation term $ \bfu^* = \bfu - \bar{\bfu}$, where $ \baru$ denotes the spatially averaged representation of $ \bfu$. Analogous to what was done for the continuous differential filter $G$ \cite{germano1986differential,grisvard1985elliptic} and discrete differential filter $ G_h$ \cite{kaya2012convergence},  we define the ROM differential filter as follows: For $ \bfu \in \bfX$ 
and a given filter width $\delta > 0$, we let $ G_r: \bX \to \bX $ 
be defined by $ G_r(\bfu) := \baru^r$, where
$ \baru^r \in \bX$ is the unique solution of the following variational problem: 
\al \label{eq: rom filter}
\delta^2\pare{\Grad\baru^r,\Grad\vr} + \pare{\baru^r,\vr} = (\bfu,\vr), \; \forall~\vr \in \bX_r.
\eal
We note that, when the ROM basis is generated by using the POD strategy, the ROM basis functions (and, thus, the ROM solution) inherit the weakly divergence-free property from the FOM. 
Leveraging this fact, we do not need to use a Stokes filter, which has been used for weekly preservation of incompressibility \cite{connors2010convergence,manica2011enabling}, 
and 
utilize instead $ G_r$ as defined in equation \eqref{eq: rom filter}.
\begin{lemma}(ROM Filtering Error Estimate \cite{xie2018numerical}) \label{lemma: Rom_Filter_Error}
    For $ \bfu^n \in \bX$, $ n = 0, 1, \dots , M, $ the ROM filter $ G_r$ satisfies 
    \al
        \delta^2& \norm{\Grad\pare{\bfu^n - G_r(\bfu^n)}}^2 + \norm{\pare{\bfu^n - G_r(\bfu^n)}}^2 
        \leq C \pare{ {N^{-2k-2} + \delt^6}  + \Ll2} + C \delta^4 \nonumber \\
        & + C\delta^2\pare{ {N^{-2k}} + \norm{S_r}_2 {N^{-2k-2}} + (1 + \norm{S_r}_2 ){\delt^6} + \Lh10  }.
    \eal
\end{lemma}

\begin{remark}
    The proof of Lemma \ref{lemma: Rom_Filter_Error} follows along the same lines as the proof of Lemma 4.3 in \cite{xie2018numerical}. The main difference is that one needs to use SEM estimates (Assumption \ref{as: FEMassumption}) instead of FEM estimates.
    Furthermore, as pointed out in~\cite[Remark 4.1]{xie2018numerical}, since the $H^1$ stability of the $L^2$ projection is not available in a ROM setting, the better $\delta$ scalings of the $H^1$ seminorm of the filtering error in~\cite{dunca2013mathematical} cannot be extended to the ROM setting in a straightforward manner.    
\end{remark}

It is easy to check that $ G_r$  is symmetric and semi-positive definite. 
The operator is also compact and $ \norm{G_r} \leq 1$.
Its associated eigenvalues satisfy $ 0 \leq \lambda_j \leq 1.$ It is also easy to check that $ (I-G_r)$ is symmetric, semi-positive definite, and compact. Moreover, its eigenvalues are $ 0 \leq 1-\lambda_j \leq 1$ so that $ \norm{(I-G_r)} \leq 1$.
Finally, it is easy to check that $ G_r \bfu = \bfu$ implies $ \bfu = 0 $,  so that $ \lambda_j = 1$ is not an eigenvalue of $G_r$,  and then $(I-G_r)$ is strictly positive definite. This allows us to define the 
norm:
\al \label{eq: TRM_norm}
        \norm{\bfphi}_* = \sqrt{ \pare{(I-G_r)\bfphi,\bfphi} }.
    \eal
More details on this can be 
found in \cite[Lemma 2.1]{ervin2007numerical}.

\section{Stability and Error Bounds}
\label{sec: main}

In this section, we formally introduce our fully discrete 
TR-ROM. First, in Lemma \ref{lemma: Stability}, we prove unconditional stability of the new TR-TOM. Then, in Theorem \ref{thm:ErrorEstimate},  we prove an {\it a priori} error bound for the TR-ROM. Finally, in Section \ref{ssec: Parameter Scalings}, we leverage the error bound in Theorem \ref{thm:ErrorEstimate} to prove parameter scalings for the TR-ROM relaxation parameter, $\chi$.

The fully discrete formulation of the TR-ROM is as follows:
For $ n= 0, 1, \dots, M-2, M-1$, find $ \ur^{n+1} \in \bX_r $ satisfying  
\al
\frac{1}{\delt}(\ur^{n+1}- \ur^n,\vr) &+ b^*( \ur^{n+1}, \ur^{n+1},\vr) + \nu(\nabla  \ur^{n+1},\nabla \vr) \nonumber\\
&+ \chi \pare{(I - G_r) \ur^{n+1} ,\vr }
 =  (\bff (t^{n+1}),\vr ), \quad \forall~ \vr \in \bX_r.
 \label{eq: ROM_TRM}
\eal
We assume that $\ur^0 $ is the $L_2$ projection of $\bfu^0$ into $ \bX_r$ i.e., $ \ur^0 = P_r(\bfu^0)$.

To prove the TR-ROM's unconditional stability in Lemma \ref{lemma: Stability}, we 
adapt the approach in \cite{belding2022efficient,ReyesGSM,LeoEMAC,neda2012increasing} to the ROM setting. 

\begin{lemma}\label{lemma: Stability}
The solution to the TR-ROM given by \eqref{eq: ROM_TRM} is unconditionally stable: For any $\Delta t>0$, the solution satisfies:
\begin{equation}
   || \ur^{M}||^2  + \nu |||\nabla \ur^{n+1}|||_{2,0}^2 + 2\chi\delt \sum_{n=0}^{M-1}|| \ur^{{n+1}}||_*^2
   \leq C_{s,r}:=||\bfu^0||^2 
   + \frac{\delt}{\nu} \sum_{n=0}^{M-1}||\bff^{n+1}||_{-1}^2. \label{stability_bound}
\end{equation} 
\end{lemma}

\begin{proof}
Choosing $\vr= \ur^{n+1}$ in \eqref{eq: ROM_TRM} yields
\als
\pare{\frac{\ur^{n+1} - \ur^n}{\delt}, \ur^{n+1}} + \nu \pare{\Grad\ur^{n+1},\Grad\ur^{n+1}} + \chi \pare{(I - G_r) \ur^{n+1} ,\ur^{n+1} } = (\bff^{n+1},\ur^{n+1})
\eals
since the skew-symmetric nonlinear term  vanishes. 
After using Cauchy-Schwarz and Young's inequalities, the dual norm of $ \bff$, and \eqref{eq: TRM_norm}, we have
\als
\frac{1}{2\Delta t} \left( \|  \ur^{n+1} \|^2 - \| \ur^n \|^2 \right) + \nu \| \nabla \ur^{n+1} \|^2 + \chi \norm{\ur^{{n+1}}}_*^2 &\leq \norm{\bff(t^{n+1})}_{-1}\norm{ \Grad \ur^{n+1}} \\
&\leq   \frac{\nu}{2} \| \nabla \ur^{n+1} \|^2  + \frac{\nu^{-1}}{2} \| \bff(t^{n+1}) \|_{-1}^2.
\eals
Rearranging some terms and multiplying by $ 2 \delt $, we obtain
\als
\left( \| \ur^{n+1} \|^2 - \| \ur^n \|^2 \right) + \nu\delt \| \nabla \ur^{n+1} \|^2 + 2\chi \delt \norm{\ur^{{n+1}}}_*^2 &\leq \frac{\delt}{\nu}  \| \bff(t^{n+1}) \|_{-1}^2 .
\eals
Because $\ur^0 $ is the $L_2$ projection of $\bfu^0$ onto $ \bX_r$, summing over time steps yields the following stability bound:
\begin{multline}
   \| \ur^{M} \|^2  +  \nu \left( \Delta t \sum_{n=0}^{M-1} \| \nabla  \ur^{n+1} \|^2 \right) + 2\chi \pare{ \delt \sum_{n=0}^{M-1} \norm{\ur^{{n+1}}}_*^2 } \\
   \leq \| \bfu^0 \|^2 + \nu^{-1}  \left( \Delta t \sum_{n=0}^{M-1} \|  \bff(t^{n+1}) \|_{-1}^2 \right). 
\end{multline}
\end{proof}

To prove the \textit{a priori} error bound in Theorem \ref{thm:ErrorEstimate}, we 
extend the strategy in \cite{belding2022efficient, neda2012increasing, LeoEMAC} to the ROM setting. 

\begin{theorem}\label{thm:ErrorEstimate}
Let $ \ur $ be the solution of the TR-ROM 
\eqref{eq: ROM_TRM}, with $\bfu$ being the true solution of NSE \eqref{eq:Strong_NSE1}-\eqref{eq:Strong_NSE2}, and let 
$\bfe^{n}= \bfu^{n}- \ur^{n}$. Under the SE Assumption \ref{as: FEMassumption}, ROM projection Assumption \ref{as: Rom Error Assumption}, 
and for sufficiently small $\delt $, we have:
\al
&  ||| \bfe|||^2_{\infty,0}+ \nu|||\nabla  \bfe^{n+1}|||^2_{2,0} + 2\chi \delt \sum_{n=0}^{M-1} \norm{ \bfe^{n+1}}_*^2 \leq CK \bigg( \nu\inv (N^{-2s-2} + \delt^2 +  \chi^2 \delta^4 ) \nonumber\\
&   + \nu\inv \chi^2  \pare{N^{-2k-2} + \delt^6 + \Ll2} 
+ \nu\inv \pare{ \mathcal{A}(N,\delt,S_r,\Ll2,\Lh10) +   \sqrt{ \Ll2} \sqrt{\Lh10}}\nonumber \\
&+  \pare{\nu + \frac{\chi^2\delta^2 + C_{s,r}}{\nu} } \bigg(  N^{-2k} + \norm{S_r}_2 N^{-2k-2} +  (1 + \norm{S_r}_2 ) \delt^6 
+ \Lh10 \bigg) \Bigg), \label{eq:TRROM_Error}
\eal
where $K$ depends exponentially on $\nu^{-3}$,  
$C$ depends on $\bfu$, $\Grad\bfu$, $\Delta \bfu$, $\bfu_t$, $\bfu_{tt}$, $p$, $C(\Omega)$, but not on $ \delt$, $N$, $\norm{S_r}_2$, $\nu$, $\delta$ or $ \chi$, 
and $ \mathcal{A}$ is defined in \eqref{eq: A}. 
\end{theorem}

\begin{proof}
First, in \eqref{eq: skew} we introduce a weak formulation of the NSE:
Find $(\bfu,p)\in \bX \times Q $ satisfying for all $\vr \in \bX_r $
\al \label{eq: skew}
& \pare{\frac{\bfu^{n+1} - \bfu^n }{\delt},\vr} + b^*(\bfu,\bfu,\vr)  + \nu (\nabla \bfu,\nabla \vr) \textcolor{black}{+ \chi ((I-G_r)\bfu^{n+1},\vr ) } 
 \nonumber \\
&= \textcolor{black}{ \chi ((I-G_r)\bfu^{n+1},\vr ) }+  (p, \divergence \vr) + \pare{\frac{\bfu^{n+1} - \bfu^n }{\delt} - \bfu_t ,\vr} + (\bff,\bfv_r). 
\eal

We split the error in the usual way as $\bfe^{n}= \bfeta^{n}+\phir^{n}$,  where $\bfeta =\bfu-P_{r}(\bfu)$ and $\phir=P_{r}(\bfu)- \ur \, \in \bX_r$, with $P_{r}(\bfu)$ being the $L^2$ projection of $\bfu$ in $\bX_r$.
Subtract equation \eqref{eq: ROM_TRM} from the weak form of the NSE \eqref{eq: skew} evaluated at $t^{n+1}$  to obtain
\begin{eqnarray*}
 & & \frac{1}{\Delta t}(  \bfe^{n+1}- \bfe^n,\vr)
 +  b^{\ast}( \bfu^{n+1}, \bfu^{n+1},\vr) - b^{\ast}(   \ur^{n+1},  \ur^{n+1},\vr) + \nu(\nabla   \bfe^{n+1},\nabla \vr)\nonumber\\
& & \textcolor{black}{ + \chi ((I-G_r)\bfe^{n+1},\vr)}  =    \left( \frac{1}{\Delta t}(\bfu^{n+1}- \bfu^n)-\bfu_t^{n+1},\vr \right) \textcolor{black}{+ \chi ((I-G_r)\bfu^{n+1},\vr)} \nonumber\\
& & +  (p^{n+1}, \divergence \vr), \quad \forall \vr \in \bX_r.
\end{eqnarray*}
Note that the nonlinear terms can be rewritten as
\begin{eqnarray*}
& &b^{\ast}( \bfu^{n+1}, \bfu^{n+1},\vr) - b^{\ast}(   \ur^{n+1},  \ur^{n+1},\vr)
= b^{\ast}(  \bfe^{n+1}, \bfu^{n+1},\vr) - b^{\ast}(   \ur^{n+1},  \bfe^{n+1},\vr) = \\
& &b^{\ast}( \bfeta^{n+1}, \bfu^{n+1},\vr) - b^{\ast}(   \phir^{n+1}, \bfu^{n+1},\vr)
+ b^{\ast}(  \ur^{n+1}, \bfeta^{n+1},\vr) - b^{\ast}(  \ur^{n+1},  \phir^{n+1},\vr).
\end{eqnarray*}
Using the above equality, splitting the error, letting $\vr=  \phir^{n+1}$, and noting that \\
$b^{\ast}(  \ur^{n+1},  \phir^{n+1},  \phir^{n+1})=0$, we obtain
\begin{eqnarray}
& &\frac{1}{2\Delta t}\left(||  \phir^{n+1}||^2-||\phir^{n}||^2 \right) + \nu||\nabla   \phir^{n+1}||^2 {+ \chi\norm{\phir^{{n+1}}}_*^2}
= \frac{1}{\Delta t}(\bfeta^{n+1}- \bfeta^n,  \phir^{n+1})\nonumber \\
& & + \nu(\nabla \bfeta^{n+1},\nabla \phir^{n+1})+(\bfu_t^{n+1}-\frac{1}{\Delta t}(\bfu^{n+1}-\bfu^n),   \phir^{n+1}) + b^{\ast}( \bfeta^{n+1}, \bfu^{n+1},  \phir^{n+1}) \nonumber \\
& &\qquad - b^{\ast}(  \phir^{n+1}, \bfu^{n+1},  \phir^{n+1})+ b^{\ast}(   \ur^{n+1}, \bfeta^{n+1},  \phir^{n+1}) +  (p^{n+1}, \divergence \phir^{n+1})\nonumber \\
& & \qquad {+ \chi ((I-G_r)\bfeta^{n+1},\phir^{n+1}) + \chi ((I-G_r)\bfu^{n+1},\phir^{n+1}) } \label{error_1} \\
& & \qquad \leq |T_1|+|T_2|+|T_3|+|T_4|+|T_5|+|T_6|+|T_7|+|T_8|+|T_9|. \nonumber
\end{eqnarray}
We now bound the above terms. 
By Definition \ref{def: Rom_Proj}, $ (\bfeta, \phir) =0$, 
which yields 
\al
|T_1| = \frac{1}{\Delta t}\abs{(\bfeta^{n+1}- \bfeta^n,  \phir^{n+1})} = 0. \label{t1_bound}
\eal
The next five terms are all bounded using standard methods.
\begin{eqnarray}
  |T_2| & \leq & \frac{\nu}{16}||\nabla   \phir^{n+1}||^2+C\nu||\nabla \bfeta^{n+1}||^2, \label{t2_bound} \\
  |T_3| & \leq & \frac{\nu}{16}||\nabla   \phir^{n+1}||^2+C\nu^{-1}\norm{\bfu_t^{n+1}-\frac{1}{\Delta t}(\bfu^{n+1}-\bfu^n)}^2 \nonumber \\
  & \leq & \frac{\nu}{16}||\nabla   \phir^{n+1}||^2 + C \nu^{-1} \delt \int_{t^n}^{t^{n+1}} \norm{\bfu_{tt}}^2 \, dt,  \label{t3_bound} \\
   |T_4|&=&|b^{\ast}( \bfeta^{n+1}, \bfu^{n+1},  \phir^{n+1})| \leq C \sqrt{|| \bfeta^{n+1}||||\nabla \bfeta^{n+1}||} ||\nabla \bfu^{n+1}|| ||\nabla    \phir^{n+1}|| \nonumber \\
   &\leq &  \frac{\nu}{16}||\nabla   \phir^{n+1}||^2 + C\nu^{-1}|| \bfeta^{n+1}||||\nabla \bfeta^{n+1}|| ||\nabla \bfu^{n+1}||^2, \label{t4_bound}\\
  |T_5|&=&|b^{\ast}(  \phir^{n+1}, \bfu^{n+1},  \phir^{n+1})| \leq  C\sqrt{||  \phir^{n+1}|| ||\nabla   \phir^{n+1}||}||\nabla \bfu^{n+1}|| ||\nabla    \phir^{n+1}||  \nonumber \\
  & \leq &  \frac{\nu}{16}||\nabla   \phir^{n+1}||^2 + C\nu^{-3}||\nabla \bfu^{n+1}||^4 ||  \phir^{n+1}||^2, \label{t5_bound} \\
  |T_6|&=&|b^{\ast}(  \ur^{n+1}, \bfeta^{n+1},  \phir^{n+1})|\leq C\sqrt{||  \ur^{n+1}|| ||\nabla   \ur^{n+1}||}||\nabla \bfeta^{n+1}|| ||\nabla    \phir^{n+1}|| \nonumber \\
  &\leq &  \frac{\nu}{16}||\nabla   \phir^{n+1}||^2 + C\nu^{-1}||  \ur^{n+1}|| ||\nabla   \ur^{n+1}|| ||\nabla \bfeta^{n+1}||^2. \label{t6_bound} 
\end{eqnarray}

For the pressure term, since $ \phir \in \bX_r \subset \bfV_N$, $ (q_N, \divergence \phir^{n+1}) = 0$ can be subtracted and then bounded in the standard way: 
\al
 |T_7| = \abs{\left( p^{n+1} - q_{N}, \Grad \cdot \phir^{n+1} \right)  }\leq  \frac{\nu}{16} \left\|\Grad \phir^{n+1} \right\|^{2}  + C \nu\inv \inf\limits_{  q_{N} \in Q_{N}} \norm{ p^{n+1} - q_N}^2. \label{p_bound}
\eal

For $ T_8 $, we use the fact that $\norm{(I-G_r)\eta}\leq\norm{\eta}$, and Cauchy-Schwarz, Poincare-Friedrichs, and Young's inequalities: 
\al
\abs{T_8} 
\leq \chi  \norm{\bfeta^{n+1}} \norm{\phir^{n+1}} \leq \frac{\nu}{16} \norm{\Grad \phir^{n+1}}^2 + C\chi^2 \nu\inv \norm{\bfeta^{n+1}}^2.
\eal
For $ T_9 $, we use again Cauchy-Schwarz, Young's, and Poincare-Friedrichs inequalities: 
\al\label{t9_bound}
\abs{T_9} &\leq \chi \norm{(I-G_r)\bfu^{n+1}} \norm{\phir^{n+1}} \leq  \frac{\nu}{16} \norm{\Grad \phir^{n+1}}^2 + C\chi^2\nu\inv \norm{ (I-G_r)\bfu^{n+1}}^2 . 
\eal

 Substituting the 
 bounds (\ref{t1_bound})-(\ref{t9_bound}) into (\ref{error_1}), multiplying by $2 \Delta t$, summing up from $n=0$ to $M-1$, and recalling that $||\phir^0||=0$ since $ \ur^0 = P_r(\bfu^0)$, yields: 
\begin{eqnarray}
& & ||  \phir^{M}||^2 +
\nu \Delta t \sum_{n=0}^{M-1} ||\nabla   \phir^{n+1}||^2
\textcolor{black}{+ 2 \chi \delt \sum_{n=0}^{M-1} \norm{\phir^{{n+1}}}_*^2} \nonumber \\
&\leq& C \Delta t \sum_{n=0}^{M-1}\nu^{-3}||\nabla \bfu^{n+1}||^4 ||  \phir^{n+1}||^2 + C\nu \Delta t \sum_{n=0}^{M-1}||\nabla \bfeta^{n+1}||^2 \nonumber \\
&+&  C\nu^{-1}\Delta t \sum_{n=0}^{M-1}  \norm{\bfeta^{n+1}} ||\nabla \bfeta^{n+1}||||\nabla \bfu^{n+1}||^2 +C \nu\inv \delt \sum_{n=0}^{M-1}  \inf\limits_{  q_{N} \in Q_{N}} \norm{ p^{n+1} - q_h}^2\nonumber \\
&+& C\nu^{-1}\Delta t \sum_{n=0}^{M-1} || \ur^{n+1}|| ||\nabla \ur^{n+1}|| ||\nabla \bfeta^{n+1}||^2 + C\nu^{-1}\Delta t^2 \int_{0}^{T}||\bfu_{tt}||^2 dt \nonumber\\
& & \qquad + \, \, { C  \chi^2 \nu\inv \delt \sum_{n=0}^{M-1} \norm{ (I-G_r)\bfu^{n+1}}^2   }+ \textcolor{black}{ C\chi^2 \nu\inv \delt \sum_{n=0}^{M-1} \norm{\bfeta^{n+1}}^2}.  \label{error_3}
\end{eqnarray}
Next, we continue to bound the error terms on the right-hand side of \eqref{error_3}. Using the approximation properties \eqref{prop3},  \eqref{eq: L2_ROM_estimate}, and \eqref{eq: H1_ROM_estimate}, we obtain:
\begin{gather}
  C\chi^2 \nu\inv \delt \sum_{n=0}^{M-1}|| \bfeta^{n+1}||^2 \leq C \chi^2 \nu\inv  \pare{ N^{-2k-2} + \delt^6 + \Ll2  },\\
C\nu \Delta t \sum_{n=0}^{M-1}||\nabla \bfeta^{n+1}||^2 \leq C \nu   \pare{ N^{-2k} + \norm{S_r}_2 N^{-2k-2} +  (1 + \norm{S_r}_2 )\delt^6 + \Lh10 },  
\end{gather}
\begin{align}
 C \nu\inv \delt \sum_{n=0}^{M-1}  \inf\limits_{p_{N} \in Q_{N}} \norm{p^{n+1} - q_N }^{2}  &\leq   C \nu\inv  N^{-2s-2} \delt \sum_{n=0}^{M-1}\norm{p}_{s+1}^{2} \nonumber\\
 &=  C \nu\inv  N^{-2s-2} \trinorm{p}_{2,s+1}^{2} \leq  C \nu\inv N^{-2s-2},\\
C\nu^{-1}\Delta t^2 \int_{0}^{T}||\bfu_{tt}||^2 dt 
&= C\nu^{-1}\delt^2 ||\bfu_{tt}||_{2,0}^2, \leq C \nu^{-1}\delt^2.
\end{align}
For the next bound, we use
\eqref{eq: H1_ROM_estimate} and Assumption \ref{as: Rom Error Assumption}, resulting in:
\al
&C\nu^{-1} \Delta t \sum_{n=0}^{M-1} ||\bfeta^{n+1}|| ||\nabla \bfeta^{n+1}||||\nabla \bfu^{n+1}||^2 
\leq
C\nu^{-1} \delt  \sum_{n=0}^{M-1}  ||\bfeta^{n+1}|| ||\nabla \bfeta^{n+1}|| \nonumber \\
\leq&  \frac{C}{\nu} \pare{ N^{-2k-2} + \delt^6 + \Ll2 }^{1/2} \pare{ N^{-2k} + \norm{S_r}_2 N^{-2k-2} +  (1 + \norm{S_r}_2 )\delt^6 + \Lh10 }^{1/2} \nonumber \\
{\leq}&  \frac{C}{\nu}  \Biggl( N^{-2k-1} +\delt^3 N^{-k}  + \sqrt{\norm{S_r}_2}N^{-2k-2} + \sqrt{\norm{S_r}_2}N^{-k-1}\delt^3 \nonumber \\
&+ \sqrt{1+\norm{S_r}_2}N^{-k-1}\delt^3 + \sqrt{1+\norm{S_r}_2}\delt^6 + N^{-k} \sqrt{ \Ll2} +N^{-k-1} \sqrt{\Lh10} \nonumber \\
& +\delt^3 \sqrt{\Lh10} + \sqrt{\norm{S_r}_2} N^{-k-1} \sqrt{ \Ll2} + \sqrt{1+\norm{S_r}_2} \delt^3 \sqrt{\Ll2} + \textcolor{black}{ \sqrt{\Ll2} \sqrt{\Lh10} } \Biggr). \label{eq: sharper}
\eal
For notational convenience, we denote {all but the last} term in \eqref{eq: sharper} as
\begin{gather}
 \mathcal{A}(N,\delt,S_r,\Ll2,\Lh10) =  N^{-2k-1} +\delt^3 N^{-k}   + \sqrt{\norm{S_r}_2}N^{-k-1}\delt^3 \nonumber \\
+ \sqrt{1+\norm{S_r}_2}N^{-k-1}\delt^3 + \sqrt{1+\norm{S_r}_2}\delt^6 + N^{-k} \sqrt{ \Ll2} + \sqrt{\norm{S_r}_2}N^{-2k-2} \nonumber\\
+ (N^{k-1}+\delt^3) \sqrt{\Lh10} + \sqrt{\norm{S_r}_2} N^{-k-1} \sqrt{\Ll2} + \sqrt{1+\norm{S_r}_2} \delt^3 \sqrt{\Ll2}. \label{eq: A}
\end{gather}
The following term utilizes the stability result from Lemma \ref{lemma: Stability} together with the Cauchy-Schwarz inequality 
and Assumption \ref{as: Rom Error Assumption}
\al
C\nu^{-1}\Delta t \sum_{n=0}^{M-1} ||  \ur^{n+1}|| ||\nabla   \ur^{n+1}|| ||\nabla \bfeta^{n+1}||^2
\leq C C_{s,r}\nu^{-1} \pare{\delt \sum_{n=0}^{M-1}\norm{\Grad\bfeta^{n+1}}^4 }^{1/2}  \nonumber\\
\leq C C_{s,r}\nu^{-1}  \pare{ N^{-2k} + \norm{S_r}_2 N^{-2k-2}
+  (1 + \norm{S_r}_2 )\delt^6 + \Lh10}. \label{eq: 2stabeta}
\eal
Using Lemma \ref{lemma: Rom_Filter_Error}, we have the following bound: 
\al
C  \chi^2 \nu\inv \delt \sum_{n=0}^{M-1} \norm{(I-G_r)\bfu^{n+1}}^2     \leq C \chi^2 \nu\inv  \pare{  N^{-2k-2} + \delt^6 + \Ll2} + C \nu\inv \chi^2 \delta^4 \nonumber \\
+ C \chi^2 \delta^2 \nu\inv  \pare{ N^{-2k} + \norm{S_r}_2 N^{-2k-2} +  (1 + \norm{S_r}_2 )\delt^6 + \Lh10}.
\eal
Thus, using the above bounds, (\ref{error_3}) becomes 
\begin{align}
& || \phir^{M}||^2 +
\nu \Delta t \sum_{n=0}^{M-1} ||\nabla   \phir^{n+1}||^2
\textcolor{black}{+ 2 \chi \delt \sum_{n=0}^{M-1} \norm{\phir^{{n+1}}}_*^2} \nonumber \\
\leq&  C \Delta t \sum_{n=0}^{M-1}\nu^{-3}||\nabla \bfu^{n+1}||^4 ||  \phir^{n+1}||^2   +C \nu\inv (N^{-2s-2} + \delt^2 +  \chi^2 \delta^4  ) \nonumber \\
& + C \nu\inv \chi^2  \pare{N^{-2k-2} + \delt^6 + \Ll2} 
+ C \nu\inv \pare{ \mathcal{A}
+   \sqrt{\Ll2} \sqrt{\Lh10}}\nonumber \\
&+ C \pare{\nu + \frac{\chi^2\delta^2+C_{s,r}}{\nu}  }
\pare{ N^{-2k} + \norm{S_r}_2 N^{-2k-2} +  (1 + \norm{S_r}_2 )\delt^6 + \Lh10}. \qquad \label{error_4}
\end{align}
Hence, by the Gronwall inequality from Lemma \ref{discreteGronwall} with $\Delta t$ sufficiently small, i.e., 
$d_n \Delta t := C\nu^{-3}||\nabla \bfu^n||^4 \Delta t < 1$, we obtain the following result: 
\begin{align}
&  ||\bfphi^{M}||^2 
+\nu \Delta t \sum_{n=0}^{M-1} ||\nabla   \phir^{n+1}||^2 + 2 \chi \delt \sum_{n=0}^{M-1} \norm{\phir^{{n+1}}}_*^2 \nonumber \\
\leq& C \exp{\left( \Delta t \sum_{n = 0}^{
M-1} \frac{d_{n}}{(1 - \Delta t \, d_{n})} \right)} \Bigg( \nu\inv (N^{-2s-2} + \delt^2 +  \chi^2 \delta^4 )  \nonumber \\
&+  \nu\inv \chi^2 \pare{N^{-2k-2} + \delt^6 + \Ll2} +  \nu\inv \pare{ \mathcal{A}
+  \sqrt{\Ll2} \sqrt{\Lh10 }}\nonumber \\
& \ + \pare{\nu +\frac{\chi^2\delta^2+C_{s,r}}{\nu} } \bigg(  N^{-2k} + \norm{S_r}_2 N^{-2k-2} +  (1 + \norm{S_r}_2 ) \delt^6 + \Lh10 \bigg) \Bigg).  \label{error_5}
\end{align}
The triangle inequality finishes the proof. 
\end{proof}

\subsection{Parameter Scalings}
\label{ssec: Parameter Scalings}


In this subsection, we 
build upon the error bound proved in Theorem~\ref{thm:ErrorEstimate} to derive parameter scalings for the time relaxation constant, $ \chi$.
To discover the optimal choice of parameter $ \chi$, we 
extend the strategy 
used in \cite{giere2015supg} to the ROM setting. 
To this end, we consider the error bound given by the result of Theorem \ref{thm:ErrorEstimate}: 
\al
&  ||| \bfe|||^2_{\infty,0}+ \nu|||\nabla  \bfe^{n+1}|||^2_{2,0} + 2\chi \delt \sum_{n=0}^{M-1} \norm{ \bfe^{n+1}}_*^2 
\leq CK \Bigg( \nu\inv (N^{-2s-2} + \delt^2 +  \chi^2 \delta^4 ) \nonumber\\
& + \nu\inv \chi^2  \pare{N^{-2k-2} + \delt^6 + \Ll2} + \nu\inv \pare{ \mathcal{A} 
+ \sqrt{ \Ll2} \sqrt{\Lh10}}\nonumber \\
&+  \pare{\nu + \frac{\chi^2\delta^2 + C_{s,r}}{\nu} } \bigg(  N^{-2k} + \norm{S_r}_2 N^{-2k-2} +  (1 + \norm{S_r}_2 ) \delt^6 + \Lh10 \bigg) \Bigg). \label{eq: stickynote 11}
\eal
First, we note that following the classical approach (see, e.g., \cite{giere2015supg}) and attempting to minimize the whole left hand side of~\eqref{eq: stickynote 11} would result in only the trivial solution $ \chi=0$. 
Choosing $\chi = 0$, however, would result in removing the time-relaxation term, which would yield the standard G-ROM.
This would clearly {be} an impractical choice since G-ROM is notoriously inaccurate in the under-resolved regime.
Thus, we propose a different strategy and minimize only the time-relaxation term (i.e., the third term) on the LHS of~\eqref{eq: stickynote 11}.
Our choice is further motivated by~\cite{ervin2012numerical}, where it is stated that $ \norm{\bfu}_*$ measures the high frequency components of $\bfu$, which is where spurious oscillations concentrate in the under-resolved regime. 
To minimize the time-relaxation term, we drop the other two terms on the LHS of~\eqref{eq: stickynote 11}, and divide by $\chi $.
To simplify the notation of the RHS of~\eqref{eq: stickynote 11},  we define a function $F$ as follows: 
\al
F(\chi) := & \pare{\nu\inv  \pare{\delt^2 + N^{-2s-2} + \mathcal{A} + \sqrt{ \Ll2 \Lh10 } } + (\nu+\nu\inv C_{s,r}) \mathcal{H}  }\chi\inv \nonumber \\
&+ \nu\inv \pare{ \delta^4 + \mathcal{L} + \delta^2 \mathcal{H} }\chi ,  \label{eq: stickynote 44}
\eal
where 
\als
\mathcal{L} := N^{-2k-2} + \delt^6 + \Ll2, \quad
\mathcal{H} := N^{-2k} + \norm{S_r}_2 N^{-2k-2} +(1+\norm{S_r}_2)\delt^6 + \Lh10.
\eals
Taking the derivative of $F$ with respect to $ \chi$ in \eqref{eq: stickynote 44} yields  
\al
F'(\chi) 
= & -\pare{\nu\inv  \pare{\delt^2 + N^{-2s-2} + \mathcal{A} + \sqrt{ \Ll2\Lh10 } } + (\nu+\nu\inv C_{s,r}) \mathcal{H}  }\chi^{-2} \nonumber \\
&+ \nu\inv \pare{ \delta^4 + \mathcal{L} + \delta^2 \mathcal{H} }.  \label{eq: DDx SN 44}
\eal
Since $\chi >0 $, setting $F' = 0$ in~\eqref{eq: DDx SN 44} results in 
\al
\nu\inv \pare{ \Ll2 + \delta^2   \Lh10+ \delta^4} \chi^2 = \pare{\nu\inv  \pare{\delt^2  +   \sqrt{ \Ll2} \sqrt{\Lh10} } + (\nu+\nu\inv C_{s,r}) \Lh10  }. \label{eq: stickynote 33}
\eal
Solving for $ \chi$ in \eqref{eq: stickynote 33} gives the optimal parameter scaling for $\chi$:
\al
\chi = \sqrt{ \frac{\nu\inv  \pare{\delt^2 + N^{-2s-2} + \mathcal{A} + \sqrt{ \Ll2\Lh10 } } + (\nu+\nu\inv C_{s,r} ) \mathcal{H} }{ \nu\inv \pare{ \delta^4 + \mathcal{L} + \delta^2 \mathcal{H} } } }.
\label{eq:optimal_chi_gen}
\eal
\begin{remark}
    The $ \chi$ scaling in \eqref{eq:optimal_chi_gen} is dimensionless 
    since the NSE~\eqref{eq:Strong_NSE1} and the numerical analysis results (including the 
    error bound \eqref{eq: stickynote 11}) are dimensionless.
\end{remark}

\section{Numerical Results}
\label{sec: Num_Sims}

In this section, we 
perform a numerical investigation of the theoretical results
obtained in Section \ref{sec: main}. To this end, we investigate whether the
TR-ROM 
{\it a priori} error bound in Theorem \ref{thm:ErrorEstimate} is
recovered numerically.
In addition, 
for the theoretical 
parameter scaling of the time-relaxation constant
$\chi$ (\ref{eq:optimal_chi_gen}) derived in Section \ref{ssec: Parameter
Scalings}, we investigate if the time-relaxation parameter {scalings} with respect
to the filter radius, $\delta$, 
{and ROM dimension, $r$, are}
recovered numerically. 
The numerical investigation 
is performed for two test problems:
the 2D flow past a circular cylinder at Reynolds number $Re=100$ (Section~\ref{subsec: 2dfpc}), and
the 2D lid-driven cavity at Reynolds number $Re = 15,000$ (Section~\ref{subsec: 2dldc}).

The numerical investigation in this section focuses on the TR-ROM 
{\it a priori} error bound 
(\ref{eq:TRROM_Error}) in Theorem \ref{thm:ErrorEstimate}, which depends on the parameters $N$, $\delt$, $\chi$, and $\delta$, as well as the ROM truncation errors $\Ll2 = \sum\nolimits^R_{j=r+1}\lambda_j$ and $\Lh10 = \sum\nolimits^R_{j=r+1} \|\nabla \phi_j\|^2 \lambda_j$ defined in (\ref{eq: POD_Proj_err}) and (\ref{eq: H10_POD_Proj_err}). 
In our numerical investigation, to measure the TR-ROM accuracy, we use the mean squared 
errors defined below:
\begin{align}
    \epl2 = \frac{1}{M+1}\sum^M_{k=0}\| {P}_{R}\bu^k_N-\bu^k_r\|^2
    ,\quad 
    \eph10 = \frac{1}{M+1}\sum^M_{k=0}\|\nabla({P}_{R}\bu^k_N-\bu^k_r)\|^2,
    \label{eq:error_def}
\end{align}
where $\bu^k_r$ is the TR-ROM 
approximation, and $P_R$ is the ROM $L^2 $ projection 
(Definition \ref{def: Rom_Proj}) 
onto the $R$-dimensional reduced space.  Specifically, we numerically investigate the rates of the convergence of $\epl2$ and $\eph10$ with respect to the ROM truncation errors $\Ll2$ and $\Lh10$, respectively.

\begin{remark}
    The TR-ROM errors in (\ref{eq:error_def}) are computed with respect to the projected FOM solution ${P}_{R}\bu^k_N$ because (i) the considered model problems do not have exact solutions, and (ii) the cost for computing the errors with respect to the projected FOM solution is independent of the number of FOM degrees of freedom, $\cN$. Measuring the error with respect to the FOM solution would require a post-processing step. The $R$ value is selected so that the error between $\bu^k_N$ and ${P}_R\bu^k_N$ is small.
\end{remark}

\subsection{TR-ROM 
Computational Implementation}

The fully discrete formulation TR-ROM (\ref{eq: ROM_TRM}) is equivalent to the following system:
\al
\frac{1}{\delt}(\ur^{n+1}- \ur^n,\vr) & + b^*( \ur^{n+1}, \ur^{n+1},\vr) + \nu(\nabla  \ur^{n+1},\nabla \vr)  \nonumber
 \\ & + \chi \pare{\ur^{n+1} - \baru^{n+1}_r ,\vr }
   =  (\bff (t^{n+1}),\vr ) \label{eq: ROM_TRM_1} \\
 \baru^{n+1}_r & = G_r(\ur^{n+1}).
 \label{eq: ROM_TRM_2}
\eal
Equations~\eqref{eq: ROM_TRM_1}--\eqref{eq: ROM_TRM_2} are equivalent to the following algebraic system:
\al
\frac{1}{\delt} M\uur^{n+1} & + (\uur^{n+1})^T B \uur^{n+1} + \nu A \uur^{n+1}  + \chi M (\uur^{n+1} - \buur^{n+1})  =  \uf^{n+1} 
\label{TR-ROM-algebraic-system-1} \\ 
\buur^{n+1} & = (M + \delta^2 A)^{-1}M\uur^{n+1},
\label{TR-ROM-algebraic-system-2} 
\eal
where $M_{ij} = (\bphi_i, \bphi_j)$, $A_{ij} = (\bphi_i, \bphi_j)$, and $B_{ijk} =  b^*(\bphi_j,\bphi_k,\bphi_i)$ are the reduced mass, stiffness, and advection operators. 
$\uf$ is the forcing vector projected onto the reduced space, and $\delta$ is the filter radius. The algebraic system \eqref{TR-ROM-algebraic-system-1}--\eqref{TR-ROM-algebraic-system-2} can be further expressed in 
a matrix-vector form:
\al
\begin{bmatrix}
    L & -\chi M \\ \huge0 & I
\end{bmatrix} 
\begin{bmatrix}
   \uur^{n+1}  \\ \buur^{n+1}
\end{bmatrix}
+ 
\begin{bmatrix}
   \uur^{n+1}  \\ \buur^{n+1}
\end{bmatrix}^T   
\begin{bmatrix}
    B & \huge0 \\ \huge0 & \huge0    
\end{bmatrix}
\begin{bmatrix}
   \uur^{n+1}  \\ \buur^{n+1}
\end{bmatrix} = 
\begin{bmatrix}
    \uf^{n+1} \\ (M + \delta^2 A)^{-1} M \uur^{n+1}   
\end{bmatrix},
\label{TR-ROM-linear-system}
\eal
where $L$ is a linear operator defined as
$L = \frac{1}{\delt} M + \nu A + \chi M.$
The 
nonlinear system \eqref{TR-ROM-linear-system} is of size $2r \times 2r$. However, using the relation,
$\buur^{n+1} = (M+\delta^2 A)^{-1} M\uur^{n+1},$
the 
nonlinear system \eqref{TR-ROM-linear-system} can be simplified to the following $r \times r$ nonlinear system: 
\al
L \uur^{n+1} - \chi M (M + \delta^2 A)^{-1} M \uur^{n+1} +  (\uur^{n+1})^T B \uur^{n+1} = \uf^{n+1},\label{eq:TR-ROM}
\eal
where the low-dimensional ($r \times r$) matrix $(M+\delta^2 A)^{-1}$ can be precomputed. 
We note that here we assumed that the zeroth mode, $\bphi_0$, is a zero velocity field, 
but the conclusion of simplifying the $2r \times 2r$ nonlinear system to size of $r \times r$ 
still holds if one has a nontrivial zeroth mode. 

We use the open-source code NekROM \cite{NekROM} to construct and solve the TR-ROM defined in (\ref{eq:TR-ROM}) for the two test problems described below. 
We mention that, in the current TR-ROM implementation, the convection term is not in the skew-symmetric form $b^*$.
We note, however, that using the standard trilinear form $b$ does not have a significant impact on the code's numerical stability.

\subsection{2D Flow Past a Cylinder}
\label{subsec: 2dfpc}


Our first 
test problem is the 2D flow past a cylinder at 
Reynolds number $\rm Re=100$, which is a canonical test case for ROMs. 
The computational domain is $\Omega = [-2.5 \, D,17 \, D] \times [-5 \, D,5 \, D]$, where $D$ is the cylinder diameter, and 
the 
cylinder is centered at $[0,0]$.

The reduced basis functions $\{\bphi_i\}^r_{i=1}$ are constructed by applying the POD procedure to 
$K=2001$ snapshots $\{\bu^k := \bu(\bx,t^k)-\bphi_0\}^K_{k=1}$. The snapshots are  
collected in the time interval $[500,~520]$ (measured in convective time units, $D/U$, where $U$ is the free-stream velocity),
after the von Karman vortex street  is developed, with sampling time $\Delta t_s=0.01$. The zeroth mode, $\bphi_0$, is set to be the FOM velocity field at $t=500$. TR-ROM is simulated on the same time interval where the snapshots are collected. 
Thus, we are in the reproduction regime.
As the initial condition for the TR-ROM, we choose the zero vector.

\subsubsection{Rates of Convergence}
We first investigate the rates of 
convergence of $\epl2$ and $\eph10$ with respect to $\Ll2$ and $\Lh10$ in the reproduction regime. 
To this end, we fix 
$N=12,~s=0,~k=1,~\delt = 2\times 10^{-3},~\delta=0.04,~\chi=0.2$, 
and vary $r$. 
Table~\ref{table:2dcyl_mag_term} shows the magnitude of each term on the right-hand side of the theoretical error estimate \eqref{eq:TRROM_Error} with these 
parameter values. 
Thus, the theoretical error estimate \eqref{eq:TRROM_Error} yields the following rates of convergence:
\begin{equation}
    \epl2  \sim \mO(\Ll2), 
    \quad
    \eph10  \sim \mO(\Lh10). \label{eq:h10_scale}
    \vspace{-0.75cm}
\end{equation}
\begin{table}[!ht]
\caption{Magnitude of each term on the right-hand side of the theoretical error estimate \eqref{eq:TRROM_Error} with respect to the number of modes, $r$. Here $N=12$, $\delt = 2 \times 10^{-3}$, $\delta=0.04$, and $\chi=0.2$. With these 
$r$ values, the matrix $2$-norm of the POD stiffness matrix $\|S_r\|_2=\mO(1)~ \text{--}~\mO(10^2)$.}
\label{table:2dcyl_mag_term}
\centering
\resizebox{\columnwidth}{!}{%
\begin{tabular}{|c|c|c|c|c|c|c|c|c|c|c||}
\hline
$r$ & $N^{-2s-2}$ & $\Delta t^{2}$ & $\chi^2\delta^4$ & $\chi^2 N^{-2k-2}$ & $\chi^2\Ll2$ & $\sqrt{\Ll2\Lh10}$ & $N^{-2k}$ & $\|S_r\|_2 N^{-2k-2}$  & \cellcolor{red!20}$\Lh10$ \\
\hline
2 &6.94e-03 &4.00e-06 &1.02e-07 &1.93e-06 &2.47e+01 &6.04e+03 &6.94e-03 &1.47e-04 &\cellcolor{red!20}5.91e+04 \\
\hline
5 &6.94e-03 &4.00e-06 &1.02e-07 &1.93e-06 &4.81e+00 &2.63e+03 &6.94e-03 &8.47e-04 &\cellcolor{red!20}5.73e+04 \\
\hline
10 &6.94e-03 &4.00e-06 &1.02e-07 &1.93e-06 &3.05e-02 &1.43e+01 &6.94e-03 &2.30e-03 &\cellcolor{red!20}2.69e+02 \\
\hline
20 &6.94e-03 &4.00e-06 &1.02e-07 &1.93e-06 &4.19e-05 &1.35e-02 &6.94e-03 &9.34e-03 &\cellcolor{red!20}1.73e-01 \\
\hline
\end{tabular}
}
\end{table}
\newpage

The behavior of the TR-ROM approximation errors $\epl2$ and $\eph10$ with respect to $\Ll2$
and $\Lh10$, respectively,
is shown in Figs.~\ref{fig:2dcyl_twrom_l2_scaling}--\ref{fig:2dcyl_twrom_h10_scaling}. In these plots, the ranges of values for $\Ll2$ and $\Lh10$ correspond
to the chosen range of values for $r$, i.e., from $r=2$ to $r=8$. 
The linear regression in the figures 
yields the following TR-ROM approximation error rates of convergence with respect to $\Ll2$ and $\Lh10$:
\begin{align}
    \epl2 \sim \mO((\Ll2)^{0.9316}), \quad
    \eph10 \sim \mO((\Lh10)^{0.9753}). 
\end{align}
Thus, the theoretical rates of convergence 
(\ref{eq:h10_scale}) 
are numerically recovered. 

\begin{figure}[!ht]
    \centering
    \begin{subfigure}{0.49\columnwidth}
    \caption{}
    \label{fig:2dcyl_twrom_l2_scaling}
    \includegraphics[width=1\textwidth]{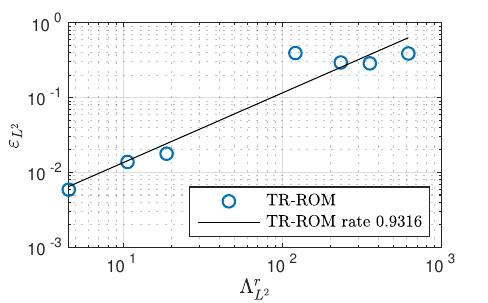}
    \end{subfigure}
    \begin{subfigure}{0.49\columnwidth}
    \caption{}
    \label{fig:2dcyl_twrom_h10_scaling}
    \includegraphics[width=1\textwidth]{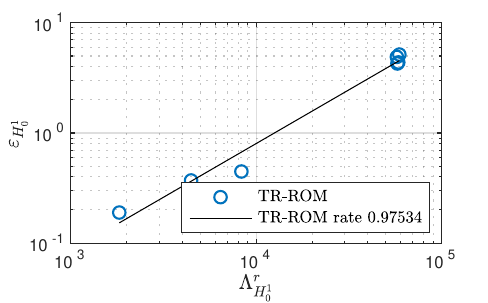}
    \end{subfigure}
    \caption{2D flow past a cylinder at $\rm Re=100$, TR-ROM with $\chi=0.2$ and $\delta=0.04$. (a) Behavior of the mean squared $L^2$ error  
    $\epl2$ with respect to $\Ll2$, and (b) behavior of the mean squared $H^1_0$ error $\eph10$ with respect to $\Lh10$. 
    } 
    \label{fig:2dcyl_err_conv_rate}
\end{figure}

\subsubsection{Scaling of $\chi$ with respect to $\delta$}

In Section \ref{ssec: Parameter Scalings}, a theoretical formulation for the
time-relaxation constant $\chi$ (\ref{eq:optimal_chi_gen}) is derived. With
$N=12,~s=0,~k=1$, and $\delt = 2\times 10^{-3}$, the 
terms 
$\delt^2,~N^{-2s-2},~N^{-2k},~\|S_r\|_2N^{-2k-2},~(1+\|S_r\|_2)\delt^6,~\text{and}~\mathcal{A}$
are relatively small. Hence, (\ref{eq:optimal_chi_gen}) can be further simplified as follows: 
\al
\chi =\sqrt{\frac{\pare{\sqrt{\Ll2\Lh10} + C_{s,r} \Lh10}}{\pare{\Ll2 + \delta^2   \Lh10 + \delta^4}}}\label{eq:opt_chi}.
\eal
Given an $r$ value, $C_{s,r}$, $\Ll2$, $\sqrtLl2h10$, and $\Lh10$ are fixed. Hence, (\ref{eq:opt_chi}) indicates that the theoretical $\chi$, $\thechi$, 
scales like 
either $\delta^{-1}$ or $\delta^{-2}$, depending on the $\delta$ value. That is, there exist two $\delta$ values, $\delta_1$ and $\delta_2$, such that
\begin{align}
   \thechi \sim \mO(1) \quad & \forall~ \delta < \delta_1, \\
   \thechi \sim \mO(\delta^{-1}) \quad & \forall~ \delta_1 \le \delta  < \delta_2, \\
   \thechi \sim \mO(\delta^{-2}) \quad & \forall~ \delta_2 \le \delta.
\end{align}
We investigate whether the scaling of the effective $\chi$, $\efchi$, with
respect to the filter radius, $\delta$, follows the scaling indicated by
(\ref{eq:opt_chi}). 

In Fig.~\ref{fig:2dcyl_chi_delta_scaling}, the behavior of 
$\thechi$ in (\ref{eq:opt_chi}) and $\efchi$ with respect to the filter radius, 
$\delta$, is shown for $r=2~\text{and}~3$. $\efchi$ is found by solving the TR-ROM and is defined to be the largest $\chi$ value that yields an accuracy that is similar to (i.e., within $5\%$ of) the accuracy for the optimal $\chi$, which is defined to be the $\chi$ value that yields the smallest $\eph10$. Specifically, for each $r$ value, we consider $18$ $\delta$ values from $[0.02, 1]$. For each $\delta$ value, $35$ $\chi$ values from $[0.001, 5]$ are considered, and $\efchi$ is selected from the $35$ $\chi$ values. 
\begin{figure}[!ht]
    \centering
    \begin{subfigure}{0.49\columnwidth}
     \includegraphics[width=1\textwidth]{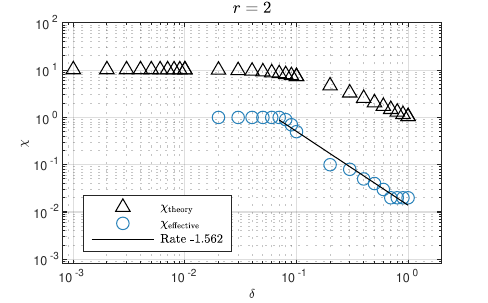}
    \end{subfigure}
    \begin{subfigure}{0.49\columnwidth}
    \includegraphics[width=1\textwidth]{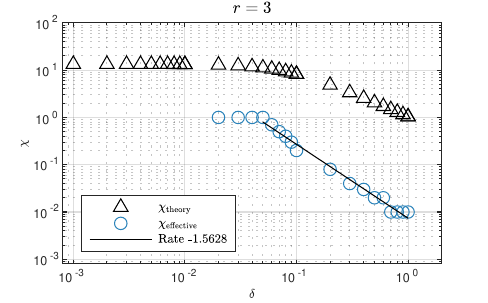}
    \end{subfigure}
    \caption{2D flow past a cylinder at $\rm Re=100$. Behavior of $\thechi$ (\ref{eq:opt_chi}) and $\efchi$ with respect to the filter radius $\delta$ for $r=2~\text{and}~3$. 
    }
    \label{fig:2dcyl_chi_delta_scaling}
\end{figure}

The results show that, for both $r$ values,  $\efchi$ scales like a constant for $\delta \le \delta_1$, 
where $\delta_1$ varies with the $r$ value. The linear regression in the figure indicates that $\efchi$ scale like $\delta^{-1.56}$ for $\delta > \delta_1$.





Next, we use two $\delta$ values to estimate the 
ratio between the effective $\chi$ and the theoretical $\chi$ (\ref{eq:opt_chi}), 
and demonstrate that 
this ratio can be used with $\thechi$ to \textit{predict} 
$\efchi$ at other $\delta$ values. 
To this end, we compute the ratio between $\thechi$ and $\efchi$ at $\delta=0.2$ and $0.3$, and 
take the average of the two ratios. The extrapolated $\chi$ values at $\delta=0.4,0.5,0.6,0.7$ are then computed using $\thechi$ and the 
average of the two ratios calculated above. 
From the results shown in Fig.~\ref{fig:2dcyl_chi_delta_scaling_extrapolated}, the extrapolated $\chi$ is close to the effective $\chi$,
{which illustrates the predictive capabilities of the theoretical parameter scaling for $\chi$ in (\ref{eq:opt_chi}).}
\begin{figure}[!ht]
    \centering
    \begin{subfigure}{0.49\columnwidth}
     \includegraphics[width=1\textwidth]{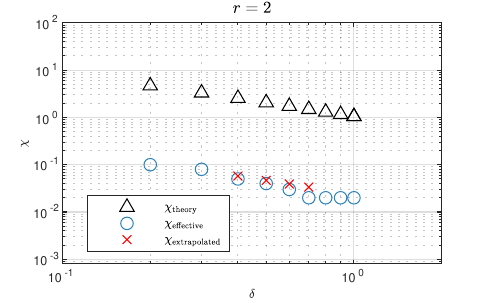}
    \end{subfigure}
    \begin{subfigure}{0.49\columnwidth}
    \includegraphics[width=1\textwidth]{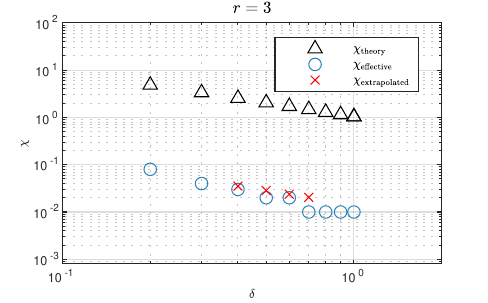}
    \end{subfigure}
    \caption{2D flow past a cylinder at $\rm Re=100$. Behavior of the extrapolated $\chi$ with respect to the filter radius, $\delta$, for $r=2~\text{and}~3$. Extrapolation is done by using the two values of $\thechi$ (\ref{eq:opt_chi}) and $\efchi$ at $\delta=0.2$ and $\delta=0.3$.
    }\label{fig:2dcyl_chi_delta_scaling_extrapolated}
\end{figure}

\subsection{2D Lid-Driven Cavity}
\label{subsec: 2dldc}
Our next example is the 2D lid-driven cavity (LDC) problem at $\rm Re=15,000$, which is a more challenging model problem compared to the 2D flow past a cylinder. As demonstrated in \cite{fick2018stabilized}, the problem requires more than $60$ POD modes for G-ROM to accurately reconstruct the solutions and QOIs. A detailed description 
of the FOM setup for this problem can be found in \cite{kaneko2020towards}. 

The reduced basis functions $\{\bphi_i\}^r_{i=1}$ are constructed by applying POD
to $K=2001$ statistically steady state snapshots $\{\bu^k :=
\bu(\bx,t^k)-\bphi_0\}^K_{k=1}$. The snapshots are collected in the time
interval 
$[6000,~ 6040]$ with sampling time $\Delta t_s=0.02$. The zeroth
mode, $\bphi_0$, is set to the FOM velocity field at $t=6000$. 

\subsubsection{Rates of Convergence}
We first investigate the rates of 
convergence of $\epl2$ and $\eph10$ with respect to $\Ll2$ and $\Lh10$, {respectively}. 
To this end, we fix 
$N=8,~s=0,~k=1,~\delt = 10^{-3},~\delta=0.004,~\chi=0.2$, 
and vary $r$. 
Table~\ref{table:2dldc_under_mag_term} shows the magnitude of each term on the right-hand side 
of the theoretical error estimate \eqref{eq:TRROM_Error} with these choices of parameters.
\begin{table}[!ht]
\caption{Magnitude of each term on the right-hand side of the theoretical error estimate \eqref{eq:TRROM_Error} with respect to the number of modes, $r$. Here $N=8$, $\delt = 10^{-3}$, $\delta=0.06$, and $\chi=0.05$. The matrix $2$-norm of the POD stiffness matrix $\|S_r\|_2=\mO(10^3)~ \text{--}$
{$\mO(10^4)$} with the considered $r$ values.}
\label{table:2dldc_under_mag_term}
\centering
\resizebox{\columnwidth}{!}{%
\begin{tabular}{|c|c|c|c|c|c|c|c|c|c|c||}
\hline
$r$ & $N^{-2s-2}$ & $\Delta t^{2}$ & $\chi^2\delta^4$ & $\chi^2 N^{-2k-2}$ & $\chi^2\Ll2$ & $\sqrt{\Ll2\Lh10}$ & $N^{-2k}$ & $\|S_r\|_2 N^{-2k-2}$  & \cellcolor{red!20}$\Lh10$ \\
\hline
2 &1.56e-02 &1.00e-06 &3.24e-08 &6.10e-07 &4.30e-03 &2.59e+02 &1.56e-02 &2.86e-01 &\cellcolor{red!20}3.92e+04 \\
\hline
4 &1.56e-02 &1.00e-06 &3.24e-08 &6.10e-07 &2.26e-03 &1.35e+02 &1.56e-02 &3.39e-01 &\cellcolor{red!20}2.01e+04 \\
\hline
8 &1.56e-02 &1.00e-06 &3.24e-08 &6.10e-07 &1.30e-03 &7.72e+01 &1.56e-02 &3.77e-01 &\cellcolor{red!20}1.14e+04 \\
\hline
16 &1.56e-02 &1.00e-06 &3.24e-08 &6.10e-07 &5.63e-04 &3.14e+01 &1.56e-02 &4.42e-01 &\cellcolor{red!20}4.38e+03 \\
\hline
\end{tabular}
}
\end{table}
Thus, the theoretical error estimate 
{\eqref{eq: stickynote 11}} yields the rates of convergence in \eqref{eq:h10_scale}.

\begin{figure}[!ht]
    \centering
    \begin{subfigure}{0.49\columnwidth}
    \caption{}
    \label{fig:2dldc_twrom_l2_scaling_under}
    \includegraphics[width=1\textwidth]{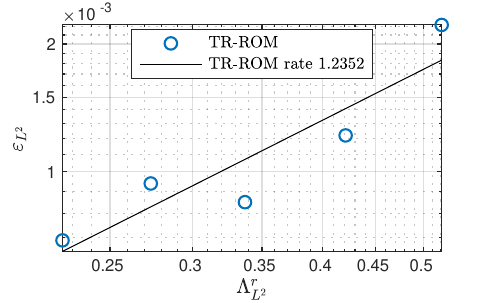}
    \end{subfigure}
    \begin{subfigure}{0.49\columnwidth}
    \caption{}
    \label{fig:2dldc_twrom_h10_scaling_under}
    \includegraphics[width=1\textwidth]{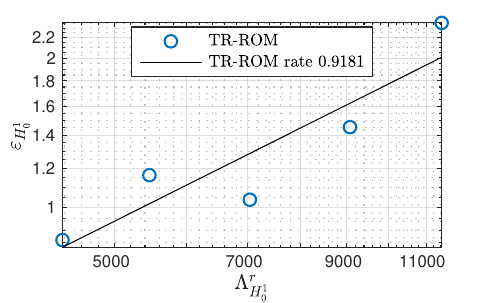}
    \end{subfigure}
    \caption{2D lid-driven cavity at $\rm Re=15,000$, TR-ROM with $\chi=0.05$ and $\delta=0.06$. 
    (a) Behavior of 
    the mean squared $L^2$ error $\epl2$ with respect to $\Ll2$, and (b) behavior of mean squared $H^1_0$ error $\eph10$ with respect to $\Lh10$. The ranges of $\Ll2$ and $\Lh10$ for the TR-ROM correspond to $r=8$ to $r=16$.} 
    \label{fig:2dldc_twrom_scaling_under}
\end{figure}

The behavior of the TR-ROM approximation errors $\epl2$ with respect to $\Ll2$ and $\eph10$ with respect to $\Lh10$, {respectively,} are shown in Figs.~\ref{fig:2dldc_twrom_l2_scaling_under}--\ref{fig:2dldc_twrom_h10_scaling_under}. The ranges of $\Ll2$ and $\Lh10$ correspond to $r=8$ to $r=16$. The linear regression in the figures indicates the following TR-ROM approximation error rates of convergence with respect to $\Ll2$ and $\Lh10$:
\begin{align}
    \epl2 \sim \mO((\Ll2)^{1.2352}), \quad
    \eph10 \sim \mO((\Lh10)^{0.9181}). 
\end{align}
Thus, the theoretical rates of convergence 
(\ref{eq:h10_scale}) 
are numerically recovered.

\subsubsection{Scaling of $\chi$ with respect to $\delta$}
With $N=8,~s=0,~k=1$ and $\delt = 10^{-3}$, the 
terms 
$\delt^2$, $N^{-2s-2}$, $N^{-2k}$, $\|S_r\|_2N^{-2k-2}$, $(1+\|S_r\|_2)\delt^6$, and $\mathcal{A}$
are {relatively} small. Therefore, (\ref{eq:opt_chi}) holds. Next, we investigate whether the scaling of 
$\efchi$ with respect to the 
filter radius, $\delta$, follows the scaling 
in (\ref{eq:opt_chi}) for $r=4,~8,~12,~\text{and}~16$.

In Fig.~\ref{fig:2dldc_chi_delta_scaling}, the behavior of 
$\thechi$ in (\ref{eq:opt_chi}) and $\efchi$ with respect to the filter
radius, $\delta$, 
is shown for 
$r=4~\text{and}~16$. The results for $r=8~\text{and}~12$ are similar.
$\efchi$ is found by solving the TR-ROM {for different parameter values. Specifically,} 
for each $r$ value, we consider $28$ $\delta$ values from $[0.001, 1]$. For each $\delta$ value, $30$ $\chi$ values from $[0.001, 1]$ are considered, and
$\efchi$ is selected from the $30$ $\chi$ values.

\begin{figure}[!ht]
    \centering
    \begin{subfigure}{0.49\columnwidth}
    \includegraphics[width=1\textwidth]{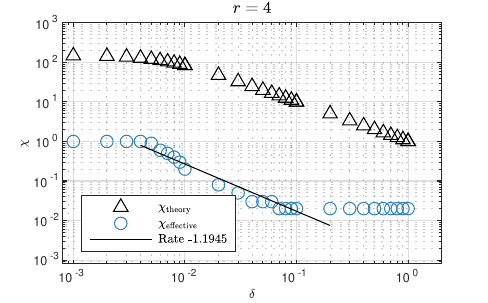}
    \end{subfigure}
    \begin{subfigure}{0.49\columnwidth}
    \includegraphics[width=1\textwidth]{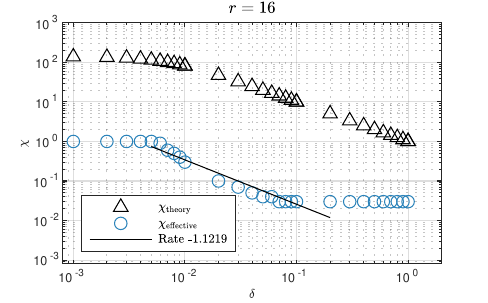}
    \end{subfigure}
    \caption{2D lid-driven cavity at $\rm Re=15,000$.
    Behavior of $\thechi$ (\ref{eq:opt_chi}) and $\efchi$  with respect to filter radius, $\delta$, for $r=4~\text{and}~16$.
    }
    \label{fig:2dldc_chi_delta_scaling}
\end{figure}

Similar to the 2D flow past a cylinder problem, the results show that 
$\efchi$ scales like a constant for certain $\delta \le \delta_1$ for all considered $r$ values, where $\delta_1$ varies with the $r$ value. The linear regression in the figure indicates that, for certain $\delta \ge \delta_1$, 
$\efchi$ scales like $\delta^{-1.2}$ for $r=4$, 
and scales like $\delta^{-1.1}$ for $r=16$. Similar scalings are observed in $r=8$ and $r=12$ cases. 
In addition, for $\delta \ge 0.07$, the filtering is too aggressive such that $G_r\bu_r \approx \mathbf{0}$. Therefore, 
the relaxation term $(I-G_r)\bu_r$ is dominated by the unfiltered solution, $\bu_r$. Hence, we see that $\efchi$ 
{does not change} for $\delta \ge 0.07$.
\begin{figure}[!ht]
    \centering
    \begin{subfigure}{0.49\columnwidth}
    \includegraphics[width=1\textwidth]{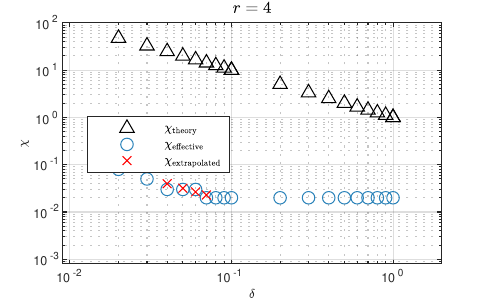}
    \end{subfigure}
    \begin{subfigure}{0.49\columnwidth}
    \includegraphics[width=1\textwidth]{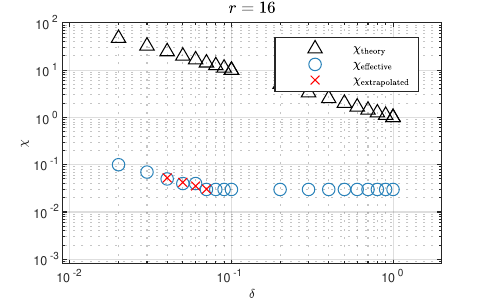}
    \end{subfigure}
    \caption{2D lid-driven cavity at $\rm Re=15,000$. 
    Behavior of the extrapolated $\chi$ with respect to the filter radius, $\delta$, for $r=4~\text{and}~16$.
    Extrapolation is done by using the two values of $\thechi$ (\ref{eq:opt_chi}) and $\efchi$ at $\delta = 0.02$ and $\delta=0.03$.
    }
    \label{fig:2dldc_chi_delta_scaling_extrapolated}
\end{figure}

Next, we use two $\delta$ values to estimate the 
ratio between $\efchi$ and $\thechi$ (\ref{eq:opt_chi}), 
and demonstrate that 
this ratio can be used with $\thechi$ to \textit{predict} 
$\efchi$ at other $\delta$ values. 
To this end, we compute the ratio between $\thechi$ and $\efchi$ at $\delta=0.02$ and $0.03$, and 
take the average of the two ratios. The extrapolated $\chi$ values at $\delta=0.04,0.05,0.06, \text{and } 0.07$ are then computed using $\thechi$ and the 
average of the two ratios calculated above. 
From the results shown in Fig.~\ref{fig:2dldc_chi_delta_scaling_extrapolated}, the extrapolated $\chi$ is close to $\efchi$, {which highlights the predictive capabilities of the theoretical parameter scalings for $\chi$ in \eqref{eq:opt_chi}}.


\subsubsection{Scaling of $\chi$ with respect to $r$}
We also investigate the scaling of 
$\efchi$ with respect to the number of modes, $r$. 
First, we note that since $\sqrt{\Ll2\Lh10} \le \Lh10$ for $r \in [2, 16]$ 
(compare columns 7 and 10 in  Table~\ref{table:2dldc_under_mag_term}) 
and $C_{s,r} = \|\bu^0\|^2$, 
which {does not depend on $r$, }
{the $\chi$} 
in (\ref{eq:opt_chi}) can be further simplified to
\al
\chi =  
\sqrt{\frac{
\Lh10}{\pare{\Ll2 + \delta^2   \Lh10 + \delta^4}}} \label{eq:opt_chi_2}.
\eal

We consider two types of filter radius for studying the scaling of $\efchi$ with respect to $r$.
First, we consider a constant filter radius, which is independent of $r$. In this case, $\thechi$ could scale like 
$\sqrt{
\Lh10/\Ll2}$, a constant related to $\delta$,  
or $\sqrt{
\Lh10}$, depending on the $r$ value. That
is, there exist two $r$ values, $r_1$ and $r_2$, such that %
\begin{align}
   \thechi \sim \mO\left(\sqrt{\frac{\Lh10}{\Ll2}}\right) \quad & \forall~ r < r_1, \\
   \thechi \sim \mO(1) \quad & \forall~ r_1 \le r < r_2,  \\
   \thechi \sim \mO\left(\sqrt{\Lh10}\right) \quad & \forall~ r_2 \le r.
\end{align}
\begin{figure}[!ht]
    \centering
    \begin{subfigure}{0.49\columnwidth}
    \includegraphics[width=1\textwidth]{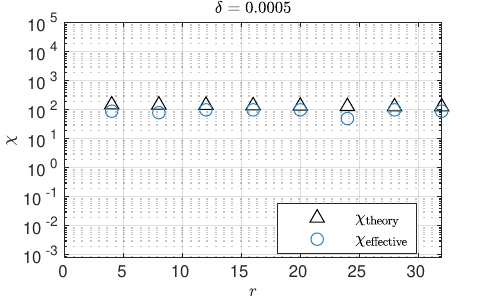}
    \end{subfigure}
    \begin{subfigure}{0.49\columnwidth}
    \includegraphics[width=1\textwidth]{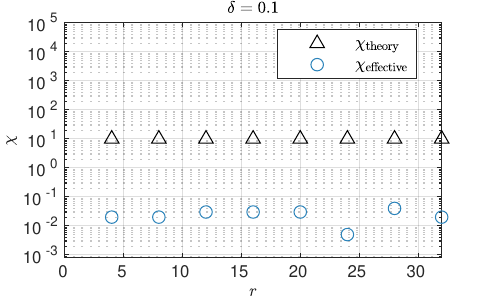}
    \end{subfigure}  
    \caption{2D lid-driven cavity at $\rm Re=15,000$ with constant filter radius. 
    Behavior of 
    $\thechi$ (\ref{eq:opt_chi_2}) and 
    $\efchi$ with respect to number of modes $r$ for $\delta=0.0005$ and $0.1$.} 
    \label{fig:2dldc_twrom_chi_vs_r_underresolved}
\end{figure}

We study the scaling of $\efchi$ with respect to $r$ for $\delta=0.0005$, $0.001$ , $0.01$, and $0.1$. In Fig.~\ref{fig:2dldc_twrom_chi_vs_r_underresolved}, the behavior of $\thechi$ 
(\ref{eq:opt_chi_2}) and $\efchi$ with respect to 
$r$ is shown for $\delta=0.0005~\text{and}~0.1$.
For $\delta=0.0005$, 
$\thechi$ scales like $\sqrt{{\Lh10}/{\Ll2}}$ for the $r$ values considered. 
Although $\sqrt{{\Lh10}/{\Ll2}}$ 
is a function of $r$, its dependency on $r$ is weak. Therefore, 
$\thechi$ behaves like a constant, as shown in the plot. 
We find 
that $\efchi$ also scales like a constant with respect to $r$. 
Specifically, 
$\efchi$ is $100$ for almost all $r$ values, except for $r=8$ and $r=24$. 
For 
$\delta = 0.1$, 
$\thechi$ scales like a constant for the $r$ values considered. 
Although $\efchi$ is not behaving like a constant, it fluctuates around 
$\chi=0.02$.
Similar behaviors are observed in $\delta=0.001$ and $\delta=0.01$ cases. 
We also note that 
$\efchi$ at $r=24$ is much smaller compared to other $\efchi$ values for all four $\delta$ values. A further investigation is required to gain a better understanding of this behavior.

The second type of filter we consider is 
the energy-based filter radius $\enedelta$ proposed in \cite{mou2023energy}: 
\begin{equation}
    \enedelta(r) = \left(\Lambda h^{2/3}+(1-\Lambda)L^{2/3}\right)^{3/2},~\text{where}~ 
    \Lambda = \sum^r_{i=1}\lambda_i/\sum^R_{i=1}\lambda_i, 
    \label{eq:energy_delta}
\end{equation}
$L$ is the characteristic length scale, and $h$ is the mesh size.
%
We emphasize that, in contrast with the constant $\delta$ case, $\enedelta(r)$ is a function of $r$. 
Substituting (\ref{eq:energy_delta}) into (\ref{eq:opt_chi_2}), we find that $\Lh10 \delta^2$ is the largest term in the denominator for $r=2$ to $r=100$. Hence, (\ref{eq:opt_chi_2}) indicates that $\thechi$ should scale like
$\delta^{-1}$ for $r \in [0, 100]$.
\begin{figure}[!ht]
    \centering
    \vspace{0.2cm}
    \includegraphics[width=0.65\textwidth]{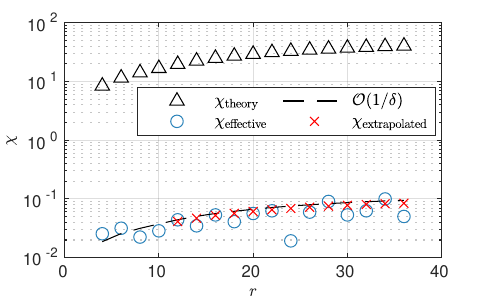}
    \caption{2D lid-driven cavity at $\rm Re=15,000$ with the energy filter radius $\enedelta$ (\ref{eq:energy_delta}).
    Behavior of 
    $\thechi$ (\ref{eq:opt_chi_2}), $\efchi$, and  extrapolated $\chi$  with respect to number of modes, $r$.
    Extrapolation 
    uses the first four values of $\thechi$ (\ref{eq:opt_chi_2}) and $\efchi$.
    }
    \label{fig:2dldc_twrom_chi_vs_r_energydelta}
\end{figure}
In Fig.~\ref{fig:2dldc_twrom_chi_vs_r_energydelta}, the behavior of 
$\thechi$ in (\ref{eq:opt_chi_2}) and $\efchi$ with respect to 
$r$ is shown. With the curve defined as $(450~\enedelta(r))^{-1}$, we clearly see that $\efchi$ also scales like $\delta^{-1}$ for the considered $r$ values, just like $\thechi$.

Next, we use four $r$ values to estimate the ratio 
between $\efchi$ and $\thechi$ (\ref{eq:opt_chi_2}), 
and demonstrate that 
this ratio can be used with $\thechi$ to \textit{predict} 
$\efchi$ at other $r$ values. 
To this end, we compute the ratio between $\thechi$ and $\efchi$ at $r=4,~6,~8$, and $10$, and 
take the average of the four ratios. The extrapolated $\chi$ {values} at $r=12,~14,\ldots,36$ are then computed using $\thechi$ and the average of the four ratio calculated above. 
From the results shown in Fig.~\ref{fig:2dldc_twrom_chi_vs_r_energydelta}, the extrapolated $\chi$ is close to $\efchi$ for all $r$ values, except 
for $r=24$.
{This highlights the predictive capabilities of the theoretical parameter scalings for $\chi$ in \eqref{eq:opt_chi_2}}.

\subsection{Parameter Scaling in the Predictive Regime}
\label{ssec:pspp}

In this section, we investigate if, in the \textit{predictive regime}, $\efchi$ still scales like $\delta^{-1}$, {given by the theoretical parameter scalings for $\chi$ in \eqref{eq:opt_chi_2}}. 
We note that 
the quantity ${\eph10} = \frac{1}{M+1} \sum^M_{k=0} \|\nabla (P_R \bu^k_N - \bu^k_r)\|^2$, which was used to determine 
$\efchi$ in the reproduction regime  (Sections~\ref{subsec: 2dfpc}--\ref{subsec: 2dldc}), is in general sensitive because it 
is based on the instantaneous error. Hence, in the predictive regime, $\efchi$  determined using $\eph10$ could be sensitive to parameters and 
deteriorate the parameter scaling.
Thus, $\eph10$ is not a suitable metric for determining $\efchi$ in the predictive regime.
Therefore, to determine $\efchi$ in both the reproduction and predictive regimes, we use an average metric, i.e., the $H^1_0$ error in the mean field,
which is defined as
\begin{equation}
\varepsilon^{avg}_{H^1_0} \equiv \|\langle \bu_r \rangle_t - \langle \bu_R \rangle_t\|^2_{H^1_0},~\text{where} \quad
\langle \bu_r \rangle_t := \sum^{{M}}_{k=0}\bu^k_r, \quad \text{and}\quad \langle \bu_R \rangle_t := \sum^{{M}}_{k=0} {P}_R\bu^k_N.
\end{equation}
The model problem is the 2D lid-driven cavity at $\rm Re=15,000$, but with a two times larger time interval, that is, $[6000, 6080]$.
We consider a larger time interval compared to the previous examples so that the quantity $\langle u_R \rangle_t$ is robust with respect to time. 

In Fig.~\ref{fig:2dldc_predictive}, the behavior of 
$\efchi$ in the predictive regime with respect to the filter radius, $\delta$, is shown for $r=4$ and $16$, along with $\thechi$ in (\ref{eq:opt_chi}) and $\efchi$ in the reproduction regime. Similar results are obtained for $r=8$ and $r=12$.
\begin{figure}[!ht]
    \centering
    \begin{subfigure}{0.49\columnwidth}
    \includegraphics[width=1\textwidth]{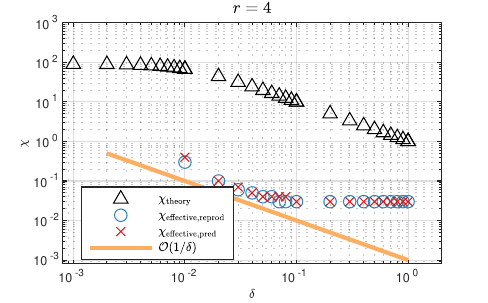}
    \end{subfigure}
    \begin{subfigure}{0.49\columnwidth}
    \includegraphics[width=1\textwidth]{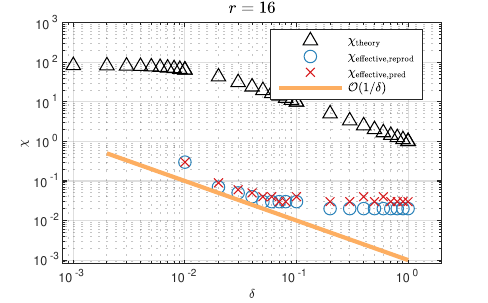}
    \end{subfigure}
    \caption{
    Behavior of $\efchi$ in predictive regime with respect to filter radius, $\delta$, for $r=4~\text{and}~16$, along with 
    $\thechi$ (\ref{eq:opt_chi}) and
    $\efchi$ in reproduction regime.
    }
    \label{fig:2dldc_predictive}
\end{figure}
To determine $\efchi$ in the reproduction regime, for each $r$ value, 
we simulate TR-ROM in the time interval $[6000, 6080]$ with $19$ $\delta$ values from $[0.01, 1]$, and $30$ $\chi$ values from $[0.001, 1]$. 
For each $\delta$ value, $\efchi$ in the reproduction regime is selected to be the largest $\chi$ value that 
yields similar accuracy (i.e., within $5\%$) as that of the optimal $\chi$, which is defined to be the $\chi$ value that yields the smallest $\epavgh10$ in the reproduction regime, 
with $M=2001$ samples. 
To determine $\efchi$ in the predictive regime, for each $r$ value, we consider the same parameter ranges for $\delta$ and $\chi$ as in the reproduction regime and simulate  
TR-ROM in a 
time interval $[6000, 6160]$, which is twice as large as the time interval in the reproduction regime. $\efchi$ in the predictive regime is selected similarly, but 
with $M=4001$ samples, where the last $2000$ samples are the data in the predictive regime.

The results show that $\efchi$ in the reproduction regime scales like $\delta^{-1}$ for the range $0.01 \le \delta \le 0.06$. 
More importantly, the results show that $\efchi$ in the predictive regime also scales like $\delta^{-1}$ in the same range as in the reproduction regime. 
In addition, for a given $\delta$ value, we observe that $\efchi$ in the predictive regime has a similar magnitude as $\efchi$ in the reproduction regime.



\section{Conclusions}
\label{sec: con}

In this work, we performed 
the first numerical analysis of the recently introduced time-relaxation reduced order model (TR-ROM) \cite{tsai2023time}. 
Specifically, we proved unconditional stability in Lemma ~\ref{lemma: Stability} and derived {\it a priori} error bounds in Theorem~\ref{thm:ErrorEstimate}. 
In addition, in Section~\ref{ssec: Parameter Scalings}, we leveraged the {\it a priori} error bounds 
to derive a formula for the time-relaxation parameter, $\thechi$, which indicates the scaling of $\chi$ with respect to the reduced space dimension $r$ and filter radius $\delta$. 
A key feature of our analysis is the coupling between the full order model (FOM) and the 
ROM, as our error bounds include terms related to the FOM discretization, which are critical for developing robust parameter scalings for the time-relaxation parameter, $\chi$. In this study, we employed the spectral element method (SEM) as the FOM, making this the first time that error bounds for SEM-based ROMs have been proven. 

In Section~\ref{sec: Num_Sims}, we demonstrated that the error convergence rate in Theorem~\ref{thm:ErrorEstimate}  and the time-relaxation parameter $\chi$ scalings with respect to $\delta$ (\ref{eq:opt_chi}) 
are recovered numerically in two test problems: the 2D flow past a cylinder and 2D lid-driven cavity. 
In addition, we estimated the ratio between the numerically found $\chi$, denoted as $\efchi$, and $\thechi$ at two $\delta$ values, and demonstrated that this ratio can be used with $\thechi$ to \textit{predict} $\efchi$ at other $\delta$ values. Furthermore, for the 2D lid-driven cavity, we demonstrated that the $\chi$ scaling with respect to $r$ (\ref{eq:opt_chi_2}) is recovered numerically for both constant filter radius and energy-based filter radius \cite{mou2023energy}.
Moreover, we showed 
that $\thechi$ can be 
also used to predict $\efchi$ at other $r$ values.
In Section~\ref{ssec:pspp}, we 
demonstrated that the $\efchi$ scaling with respect $\delta$ in the reproduction regime is also observed in the \textit{predictive} regime. In particular, we showed that $\efchi$ in the 
predictive 
regime has a similar magnitude as $\efchi$ in the reproduction regime for most $\delta$ values. This
illustrates the practical value of the 
new parameter scaling. 

For future work, there are several 
promising research directions to explore. These include 
performing numerical analysis (e.g., deriving {\it a priori} error bounds) for nonlinear filtering 
and data-driven 
extensions of the new TR-ROM 
and other regularized ROMs. 
These {\it a priori} error bounds could then be leveraged to determine new ROM parameter scalings. 
Finally, these scalings could be tested in the predictive regime of challenging numerical simulations (e.g., turbulent channel flow) to determine their range of applicability in practical settings.

\section{Acknowledgments}
This work was 
supported by the 
National Science Foundation through grants
DMS-2012253 and CDS\&E-MSS-1953113, 
and
by grant PID2022-136550NB-I00 funded by MCIN/AEI/10.13039/ 501100011033 and 
the European Union ERDF A way of making Europe. 

 \citestyle{numeric-comp}
 \bibliographystyle{elsarticle-num}
\bibliography{Ref}
\end{document}